\documentclass[letterpaper]{amsart}
\usepackage{times}
\usepackage[numbers,sort]{natbib}
\usepackage{microtype}
\usepackage{amssymb}
\usepackage{bbm}
\usepackage{cleveref}
\usepackage{csquotes}
\usepackage[normalem]{ulem}
\usepackage{subcaption}
\usepackage{graphicx}

\let \P \relax

\newcommand{\P}{\mathbb{P}}

\newcommand{\N}{\mathbb{N}}
\newcommand{\R}{\mathbb{R}}

\newcommand{\de}{\delta}
\newcommand{\la}{\lambda}
\newcommand{\si}{\sigma}

\newcommand{\minus}{\smallsetminus}
\DeclareMathOperator{\GL}{GL}

\DeclareMathOperator{\cone}{cone}

\DeclareMathOperator{\depth}{depth}
\DeclareMathOperator{\sign}{sign}
\DeclareMathOperator{\interior}{int}
\newtheorem{definition}{Definition}
\newtheorem{remark}{Remark}
\newtheorem{theorem}{Theorem}
\newtheorem{lemma}{Lemma}

\newtheorem{example}{Example}

\begin{document}

\title{The Chiral Domain of a Camera Arrangement}
\author{Sameer Agarwal}
\address{Google Inc.}
\email{sameeragarwal@google.com}
\author{Andrew Pryhuber}
\address{University of Washington, Seattle}
\email{pryhuber@uw.edu}
\author{Rainer Sinn}
\address{Freie Universit\"at, Berlin}
\email{rainer.sinn@fu-berlin.de}

\author{Rekha R. Thomas}
\address{University of Washington, Seattle}
\email{rrthomas@uw.edu}
\thanks{Pryhuber and Thomas were partially supported by the NSF grant DMS-1719538}

\begin{abstract}
We introduce the chiral domain of an arrangement of cameras $\mathcal{A} = \{A_1,\hdots, A_m\}$ which is the subset of $\P^3$ visible in $\mathcal{A}$. It 
generalizes the classical definition of chirality to include all of $\P^3$ and offers a unifying framework for studying multiview chirality. We give an algebraic description of the chiral domain which allows us to define and 
describe the chiral version of Triggs' joint image~\cite{triggs1995geometry,triggs95}.
We then use the chiral domain to re-derive and extend prior results on 
chirality due to Hartley~\cite{hartley1998chirality,HartleyZisserman2004}.
\end{abstract}

\maketitle

\section{Introduction}\label{sec:introduction} 

In computer vision, chirality refers to the constraint that for a scene point to be visible in a camera, it must lie {\em in front} of it~\cite{hartley1998chirality}. 
There is now a mature theory of multiview geometry that ignores this constraint~\cite{HartleyZisserman2004} modeling image formation only via 
algebraic constraints coming from a (projective) camera being a (rational) 
linear map from $\P^3$ to $\P^2$. Chirality imposes additional semialgebraic conditions. 

The study of chirality was initiated by Hartley 
in his seminal paper~\cite{hartley1998chirality}, much of which forms 
\cite[Chapter 21]{HartleyZisserman2004}. Faugeras and Laveau introduced the oriented projective geometry framework to model chirality constraints in vision \cite{laveau1996oriented}. Werner and Pajdla built on this framework deriving a theory of oriented matching constraints which enforce chirality \cite{werner2001oriented}.  In the case of two cameras, they gave a geometric interpretation of these constraints in the epipolar plane and suggest methods to use chirality for reducing the search space in stereo matching \cite{WernerPajdla2001}. 
Werner further showed that such orientation constraints naturally give rise to  combinatorial conditions on sets of images necessary for them to correspond to a true scene
\cite{werner2003combinatorial,werner2003constraint}.

In this paper we develop a general theory of {\em multiview chirality} for an arrangement of projective  
cameras $\mathcal{A} = \{A_1, \ldots, A_m\}$ with distinct centers. Our central contribution is the notion of the {\em chiral domain} of $\mathcal{A}$ which is the subset of $\P^3$ (the world) that 
is visible in the cameras of $\mathcal{A}$. This is a multiview generalization of the classical 
definition of chirality, and covers all of $\P^3$ including infinite points,
namely vanishing points,
and points on the principal planes of the cameras. The previous definition only covered 
finite world points. The extension is easy for one camera but subtle for multiple cameras 
as we explain. We show that the chiral domain admits a simple semialgebraic description by quadratic inequalities determined by the principal planes of the cameras and the plane at infinity. This description is the workhorse of the paper.

Recall that the {\em joint image} of $\mathcal{A}$ due to Triggs~\cite{triggs1995geometry}, 
\cite{triggs95}, is the set of all ``images'' of $\P^3$ in the cameras of  $\mathcal{A}$, ignoring chirality. Starting with the seminal work of Longuet-Higgins~\cite{longuet1981computer}, there is now a complete algebraic and set theoretic characterization of the joint image~\cite{idealsofthemultiviewvariety,M12,FLP01,HartleyZisserman2004,HA97,YM12,THP15}. The closure of the joint image is an algebraic variety in $(\P^2)^m$ called the {\em joint image variety} in \cite{THP15} and the {\em multiview variety} in \cite{idealsofthemultiviewvariety}, \cite{AST11}.

We define and describe the {\em chiral joint image} of $\mathcal{A}$, which is the analog of the joint image under the requirement of 
chirality. It is the image of the chiral domain of $\mathcal{A}$ 
and is thus the true image of the world in $\mathcal{A}$. 
The chiral joint image is a subset of the joint image.
One of our main results is a 
semialgebraic description  (using polynomial equalities and inequalities) of the chiral joint image. 
This semialgebraic description is a refinement of the multiview constraints (epipolar \& trifocal), when the world points are constrained to lie in front of the cameras. 
As a simple application, just like the joint image (epipolar constraints) is used to limit stereo matching to epipolar lines,  the chiral joint image can be used to further limit stereo matching to a subregion of the epipolar line (See \Cref{fig:cjiallregionsseperate} and~\Cref{fig:cjiallregions}).




Hartley studies the question of when a projective reconstruction of a set of point correspondences in $(\P^2)^m$ can be turned into a chiral reconstruction by applying a $\P^3$-homography. His answer is in terms of the feasibility of a system of linear inequalities called {\em chiral inequalities} \cite[Chapter 21]{HartleyZisserman2004} which boils down to solving two linear programs. Using the chiral domain we recover this result. We are also able to interpret quasi-affine transformations in this language. 

A limitation of Hartley's definition of chirality is that it only works for finite points. This is not a problem when reasoning about chiral projective reconstructions because 
for general projective cameras, if there is a 3-d reconstruction from images, then there is always a 3-d reconstruction with finite cameras and world points (a theorem of H-L. Lee [9]). This is not true for Euclidean cameras. Using the chiral domain (which defines chirality for all of $\mathbb{P}^3)$ we are able to extend Hartley's results to Euclidean reconstructions.

In \cite{hartley1998chirality}, Hartley also characterizes the existence of a chiral reconstruction for two-views in terms of a sign condition on the given projective reconstruction. Werner et. al. also study the two-view case, considering both minimal and nonminimal configurations~\cite{werner2003constraint,WernerPajdla2001}. Nist{\'e}r \& Schafflitzky consider the minimial problem in the Euclidean case~\cite{nisterschaffalitzky}. Our polyhedral approach provides a simple proof of this two-view result. In \cite{hartley1998chirality} Hartley remarks (without proof or explanation) that the result does not extend beyond two-views. We use our tools to construct a counterexample with three cameras and offer an explanation for the gap.



Hartley develops chirality using  classical projective geometry~\cite{hartley1998chirality,HartleyZisserman2004}. Other authors have used {\em oriented projective geometry} \cite{stolfi1991oriented} to model chirality (see for example, \cite{laveau1996oriented}, \cite{werner2001oriented}). This approach requires the choice of an orientation of the cameras involved based on a world point that is known to be in front of the cameras. The initial choice of orientation percolates down a chain of subsequent choices. In particular, the projective camera $A$ is considered different from $-A$ in this theory.

Following Hartley, we use classical projective geometry in this paper.
By staying in this setup we, like Hartley, are able to avoid all of the choices needed 
in oriented projective geometry, and still obtain a perfectly valid theory of chirality. 
Working in the projective framework is especially handy when describing the chiral joint image in this paper which is naturally a subset of the joint image, a quasi-projective algebraic variety. The trick is to derive meaningful inequalities in projective space, by which we mean inequalities that are invariant under scaling. 

This paper is organized as follows. In \Cref{sec:background}, we provide some background and set the notation. \Cref{sec:chirality} introduces the 
chiral domain of a camera arrangement. This is then used to define and describe the 
chiral joint image of the camera arrangement in \Cref{sec:chiraljointimage}. 
In \Cref{sec:chiral-reconstructions} we establish the connections to Hartley's results about when a projective reconstruction can be made chiral using a homography.
We also show how our results connect to quasi-affine transformations and Hartley's two-view results on chirality. The case of Euclidean reconstructions is treated in \Cref{sec:euclideanexistence}. Section~\ref{sec:summary} summarizes our contributions.
Many of the technical proofs in the three main sections can be found in~\Cref{sec:appendix}. 

{\bf Acknowledgments.} We thank Tomas Pajdla for discussions at the start of this project 
and for pointers to the chirality literature.

\section{Background and Notation}
\label{sec:background}

 The sets of nonnegative integers, nonnegative real numbers, and positive real numbers are denoted by $\N, \R_+$, and $\R_{++}$, respectively. 
 $\P^n$ denotes n-dimensional projective space over the reals, which is $\R^{n+1} \minus \{0\}$ modulo the equivalence relation $\sim$ where 
$\mathbf{x} \sim \mathbf{y}$ if $\mathbf{x}$ is a scalar multiple of $\mathbf{y}$. If $\mathbf{x} \sim \mathbf{y}$, then we say that 
$\mathbf{x}$ and $\mathbf{y}$ are equal in $\P^n$, or $\mathbf{x}$ is {\em identified with} $\mathbf{y}$. We use $=$ to denote coordinate wise equality in $\R^n$.

In multiview geometry, we focus on $\P^3, \P^2$ and $\R^3, \R^2$, where $\P^n$ is a compactificaction of $\R^n$ with respect to the embedding $\R^n\to \P^n$, $\mathbf{x}\mapsto \widehat{\mathbf{x}} = (\mathbf{x},1)$. So points whose last coordinate is nonzero are said to be \emph{finite}, whereas points whose last coordinate is $0$ form the \emph{hyperplane at infinity}. We write the plane at infinity as 
$L_\infty := \{ \mathbf{q} \in \P^3 \,:\, \mathbf{n}_\infty^\top \mathbf{q} = 0 \},
$ where we fix the normal $\mathbf{n}_\infty  = (0,0,0,1)^\top$. 

We denote points in $\P^3$ and $\R^3$ by $\mathbf{q}$ allowing the context to decide where $\mathbf{q}$ lies. Similarly we denote points in $\P^2$ and $\R^2$ by $\mathbf{p}$. 
The dehomogenization of a finite point $\mathbf{q} \in \P^3$ is denoted $\widetilde{\mathbf{q}} := (q_1/q_4, q_2/q_4, q_3/q_4)^\top$.

The {\em projectivization} of a set $S \subseteq \R^{n+1}$ is the union of all lines through the origin in $\R^{n+1}$ that intersect $S$.
For an example see \Cref{fig:projectivization}.

\begin{figure}
  \centering
  \includegraphics[width=.6\linewidth]{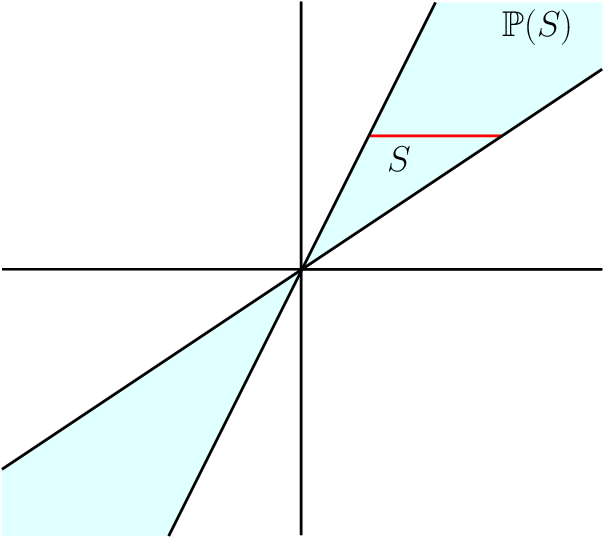} 
  \caption{Projectivization of a set $S \subset \R^2$. }
  \label{fig:projectivization}
\end{figure}

A {\em projective} camera is a matrix $A= \begin{bmatrix} G & \mathbf{t} \end{bmatrix} \in 
\R^{3 \times 4}$ of rank $3$. The camera A is {\em finite} if $\det(G) \neq 0$. The {\em center} of the camera $A$ is the unique point $\mathbf{c} \in \P^3$ such that 
$A \mathbf{c}=0$. The camera $A$ is 
finite if and only if its center $\mathbf{c}\in \P^3$ is finite.
All cameras considered in this paper are finite. For consistency, we will choose the $\R^4$ representative $\mathbf{c}_A = \begin{bmatrix} -G^{-1} \mathbf{t} \\ 1 \end{bmatrix}$ for the center of $A$ .

The {\em principal plane} of a finite 
camera $A= \begin{bmatrix} G & \mathbf{t} \end{bmatrix}$ is the hyperplane
 $L_A := \{ \mathbf{q} \in \P^3 \,:\, A_{3,\bullet} \mathbf{q} = 0 \},$
 where $A_{3,\bullet}$ is the third row of $A$, i.e. $L_A$ is the set of points in $\P^3$ that image to infinite points in $\P^2$ under camera $A$. Note that the camera center $\mathbf{c}$ lies on $L_A$. We regard $L_A$ as an oriented hyperplane in $\R^4$ with normal vector $\mathbf{n}_A := \det(G) A_{3 \bullet}^\top$, which we call the \emph{principal ray} of $A$. 
 The $\det(G)$ factor makes sure that if we pass from $A$ to $\lambda A$ for some nonzero scalar $\lambda \in \R$, the normal vector of the 
principal plane does not change sign. 

The {\em world} $\R^3$, which is to be imaged by $A$, is modeled as the affine patch in $\P^3$ with $q_4 = 1$. This allows the 
identification of a 
finite point  $\mathbf{q} \in \P^3$ with the world point $\widetilde{\mathbf{q}} \in \R^3$, and a world point $\mathbf{q} \in \R^3$ with the finite point 
$\widehat{\mathbf{q}} \in \P^3$. The image of $\mathbf{q} \in \P^3$, in the camera $A$ is  $A\mathbf{q} \in \P^2$. The rational map $A\colon \P^3 \dashrightarrow \P^2$, $\mathbf{q}\mapsto A\mathbf{q}$, is defined for all
$\mathbf{q} \in \P^3$ except the center $\mathbf{c}$ of $A$\footnote{The broken arrow ($\dashrightarrow$) and the phrase ``rational map" mean here that the domain of the map $A$ is not actually $\P^3$ but rather $\P^3\setminus\{\mathbf{c}\} $.}. 

Our theory is developed using explicit coordinate representations of geometric objects such as cameras, their principal rays, and the plane at infinity as discussed above. Throughout the paper, we favor expressions in terms of concrete vectors such as $\mathbf{n}_\infty^\top \mathbf{q}$ and $\mathbf{n}_A^\top \mathbf{q}$. We remark that these expressions are not intended to be thought of as coordinate-free.

Let $\mathcal{A} = \{A_1, \ldots, A_m\}$ denote 
an arrangement of $m$ cameras with distinct centers. We use the shorthand $\mathbf{n}_i$ for the principal ray of camera $A_i$. Given a pair of cameras $A_i,A_j$ with centers $\mathbf{c}_i$ and $\mathbf{c}_j$, let $\mathbf{e}_{ij}$  denote the image of $\mathbf{c}_j$ in $A_i$. The points $\mathbf{e}_{ij}$ are called {\em epipoles}.  The line through $\mathbf{c}_i, \mathbf{c}_j$ is 
called the {\em baseline} of the pair of cameras $\{A_i, A_j\}$.
All points on the base line (except for the centers themselves) will image in the two cameras at their respective epipoles $(\mathbf{e}_{ij}, \mathbf{e}_{ji})$.

We now recall some basics of convex geometry. Further details can be found in \cite{boyd2004convex}. The (polyhedral) {\em cone} $K_U$ spanned by a set of vectors $U = \{\mathbf{u}_1, \ldots, \mathbf{u_\ell}\}\subseteq \R^n$ is the set of all nonnegative linear combinations of the vectors in $U$, so
\begin{align*}
K_U &:=  \textup{cone}(\mathbf{u}_1, \ldots, \mathbf{u}_\ell) \\
    &= \{ x_1 \mathbf{u}_1 + \cdots +  x_\ell \mathbf{u}_\ell \ |\ x_1, \ldots, x_\ell \geq 0 \}.
\end{align*}

The \emph{dimension} of a cone $K_U \subseteq \R^n$ is the dimension of the smallest vector space containing it. Let $\interior K_U$ denote the relative interior of $K_U$, namely the 
interior of $K_U$ in the linear span of $U$. A cone is {\em pointed} if it does not contain a line.

The {\em dual cone} $K_U^\ast$ to $K_U$ is the set of all vectors that make nonnegative inner product with every vector of $K_U$, so
\begin{equation}
K_U^\ast := \{\mathbf{y}\ |\ \forall\ \mathbf{x} \in K_U \colon \mathbf{y}^\top \mathbf{x} \geq 0 \}.
\end{equation}

\begin{figure}[ht]
  \centering
  \includegraphics[width=.7\linewidth]{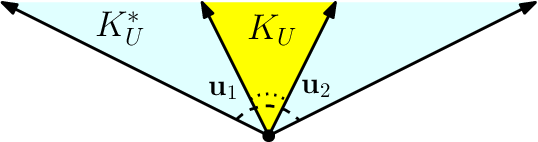} 
  \caption{A cone $K_U$ and its dual cone $K_U^\ast$. }
  \label{fig:dualcone}
\end{figure}

Polyhedral cones are bidual in the sense that $K_U = (K_U^\ast)^\ast$.
The cone $K_U$ is pointed if and only if $K_U^\ast$ is full-dimensional, and because of 
biduality, $K_U$ is full-dimensional if and only if $K_U^\ast$ is pointed. 
If $K_U^\ast$ is full-dimensional, then  
$\interior K_U^\ast = 
\{\mathbf{y}\ |\ \mathbf{y}^\top \mathbf{u}_i > 0, \,\,\forall\, i=1,\ldots, \ell \}$. Otherwise, $K_U = L + K_W$ where $L$ is a subspace and $K_W = \textup{cone}(\mathbf{w}_1, \ldots, \mathbf{w}_p)$ is a pointed cone. Writing $L^\perp$ as the kernel of a matrix $A$, 
$K_U^\ast = L^\perp \cap K_W^\ast = 
\{\mathbf{y} \,:\, A \mathbf{y} = 0, \mathbf{y}^\top \mathbf{w}_j > 0 \,\, \forall \, j=1,\ldots,p \}$. See \Cref{fig:dual of cone sum}.

\begin{figure}[ht]
  \centering
  \includegraphics[width=.9\linewidth]{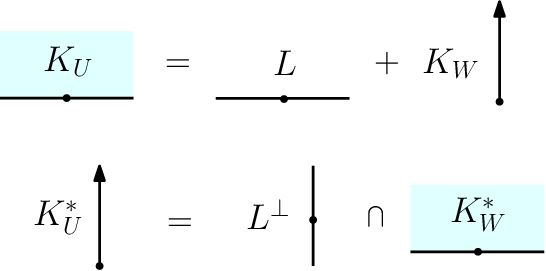} 
  \caption{The non-pointed cone $K_U$ on the top left is the sum of the subspace 
  $L$ and the pointed cone $K_W$. Its dual cone $K_U^\ast$ in the second row is the 
  intersection of $L^\perp$ and $K_W^\ast$.}
  \label{fig:dual of cone sum}
\end{figure}

\begin{remark}\label{rem:LP check}
Membership in $\interior K_U^\ast$ can be determined via linear programming. 
Indeed, the optimal value of the linear program 
$$\max \{ \varepsilon  \,:\, A\mathbf{y} = 0, \,\,\varepsilon \leq 1, \,\,
\mathbf{y}^\top \mathbf{w}_j \geq \varepsilon, \,\,j=1,\ldots,p \}$$
is positive if and only if $\interior K_U^\ast \neq \emptyset$.
\end{remark}

\section{The Chiral Domain of an Arrangement of Cameras}
\label{sec:chirality}

We begin by recalling the definition of the depth of a finite point $\mathbf{q}$ in a finite camera $A$. It is essentially the projection of $\widetilde{\mathbf{q}}-\widetilde{\mathbf{c}}$ along the principal ray, see \cite{HartleyZisserman2004}.
Formally, it is defined as
\begin{equation}
\operatorname{depth}(\mathbf{q};A) := 
\left(\frac{1}{|\det(G)|\|G_{3,\bullet}\|} \right) \frac{( \mathbf{n}_A^\top \mathbf{q})}{(\mathbf{n}_\infty^\top \mathbf{q})}. \label{def:depth}
\end{equation}
Notice that the sign of $\operatorname{depth}(\mathbf{q};A)$ is unaffected by scaling $\mathbf{q}$. 
In fact, this definition of depth as a rational function of degree $0$ in $\mathbf{q}$ (meaning that the degree of numerator and denominator are equal) underlines the inherently projective nature of the notion of depth. Furthermore, scaling $A$ also does not affect the sign of $\operatorname{depth}(\mathbf{q};A)$ because the orientation of $\mathbf{n}_A = \det(G)A_{3,\bullet} $ is independent of scaling.

We say that a finite point
$\mathbf{q}$ not on the principal plane is {\em in front of} the camera
if $\operatorname{depth}(\mathbf{q};A) > 0$~\cite{hartley1998chirality}.  Since only the sign of 
$\operatorname{depth}(\mathbf{q};A)$ matters, we define the {\em chirality} of $\mathbf{q}$ in $A$ 
to be $\chi(\mathbf{q};A) = \sign(\operatorname{depth}(\mathbf{q};A))$. This is either $1$ or $-1$. This definition of chirality excludes finite points with zero depth and points at infinity.
Also, it treats chirality  as a per camera concept.


 
 In this section we will extend the above notion of chirality to all points in $\P^3$ with respect to one or more finite cameras. This will then lead to the central 
concept of this paper, the {\em chiral domain of an arrangement of cameras}, which we use to develop a unified theory of multiview chirality.

The exclusion of points on the principal plane and the plane at infinity makes the algebraic 
treatment of chirality complicated because it forces us to work with strict inequalities 
to avoid boundary points where depth is not defined. Our generalization 
below remedies this situation and in particular leads to  
an algebraic description of the {\em chiral joint image} of a camera arrangement as a subset of the classical {\em joint image}~\cite{idealsofthemultiviewvariety,THP15}. 

We now discuss how we might decide the chirality of points on the principal plane and the plane at infinity. Let us begin by considering the case of one camera: Traveling along the line through a point $\mathbf{a}\in\R^3$ in direction $\mathbf{q}\in\R^3$ corresponding to a point $\widehat{\mathbf{q}}\in\P^3$, 
we see from the definition that chirality 
changes only when we cross the principal plane or the plane at infinity (or it is always $0$). 
So any point on either plane is arbitrarily close to finite points of chirality $1$ in the camera $A$.
More precisely, we should argue in $\P^3$ directly: Suppose $\mathbf{q} = (q_1,q_2,q_3,0)^\top \in L_\infty$ and 
$\mathbf{a} = (a_1, a_2, a_3, 1)^\top$ has positive depth in $A$. 
We may assume that $\mathbf{n}_A^\top \mathbf{q} \geq  0$ since otherwise we can work with $-\mathbf{q}$. Then the points
$\mathbf{q}_t := \frac{1}{t} \mathbf{a} + \mathbf{q}$ have positive depth for $t>0$, 
and since $\mathbf{q} = \lim_{t \rightarrow \infty} \mathbf{q}_t$, it is arbitrarily close 
to points with chirality $1$ in $A$. Now suppose $\mathbf{q} \in L_A \backslash L_\infty$. 
We may assume that 
$q_4 > 0$ since otherwise we can take $-\mathbf{q}$. As before, pick a finite point $\mathbf{a} = (a_1, a_2, a_3, 1)^\top$ with positive depth in $A$. Then the sequence of points 
$\mathbf{q}_t := \frac{1}{t} \mathbf{a} + \mathbf{q}$ have chirality $1$ for all $t > 0$ and again,  
$\mathbf{q}$ is arbitrarily close to finite points of positive depth in $A$.
Thus it seems natural to consider all points on $L_\infty \cup L_A$ to have
chirality $1$ in $A$. We could just as well argue that all points in $L_A \cup L_\infty$ should 
have chirality $-1$ with respect to the camera $A$ by a similar argument to the above. However, 
since points on $L_\infty$ are vanishing points of rays with positive depth in $A$,
our physical intuition is that $L_\infty$ is visible in $A$, so points on $L_\infty$ should have chirality 1 in $A$.

For arrangements of cameras, even just two cameras, the situation is more complicated.
Consider \Cref{fig:trackoppositeface}, where two cameras 
$A_1, A_2$ are placed on a train track looking in opposite directions so that 
$\mathbf{n}_{A_1} = - \mathbf{n}_{A_2}$. For instance, we could choose the camera facing right 
to be $A_1 = \begin{bmatrix} I & \mathbf{0} \end{bmatrix}$ and the camera facing left to be $$A_2 = 
\begin{bmatrix} 1 & 0 & 0 & 0 \\ 0 & -1 & 0 & 1 \\ 
0 & 0 & -1 & 0 \end{bmatrix}.$$
Then $\mathbf{c}_1 = (0,0,0,1)^\top$ and $\mathbf{c}_2 = (0,1,0,1)^\top$ which 
differ in the vertical direction and the principal rays drawn in $\R^3$ are 
$(0,0,1)$ and $(0,0,-1)$.
For any finite $\mathbf{q} \not \in L_{A_1} = L_{A_2}$, $\chi(\mathbf{q};A_1) = - \chi(\mathbf{q};A_2)$, and hence there is no 
sequence of finite points that each have positive depth in 
the two cameras that can approach points on $L_\infty$ or $L_{A_1} = L_{A_2}$. This is in contrast to the single camera case where all points can be approached by finite points with positive depth in the camera.

\begin{figure}[h]
    \centering
    \includegraphics[width=\linewidth]{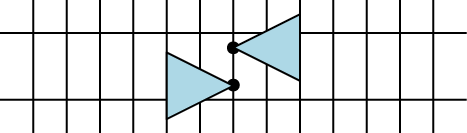}
    \caption{Cameras facing in opposite directions on train tracks.}
    \label{fig:trackoppositeface}
\end{figure}
\begin{figure*}[h]
\begin{subfigure}{.49\textwidth}
  \centering
  \includegraphics[width=\linewidth]{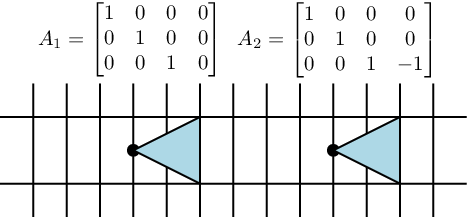} 
  \caption{}
  \label{fig:tracksameface}
  \centering
  \includegraphics[width=.49\linewidth]{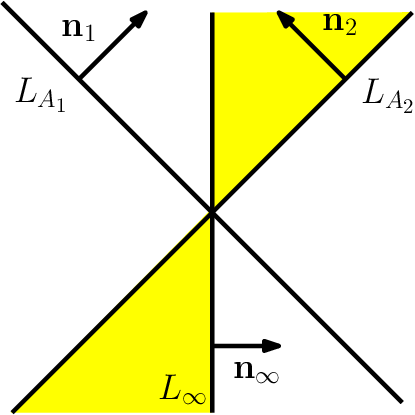}
  \caption{}
\label{fig:biginfiniteintersection}
\end{subfigure}
\begin{subfigure}{.49\textwidth}
  \centering
  \includegraphics[width=\linewidth]{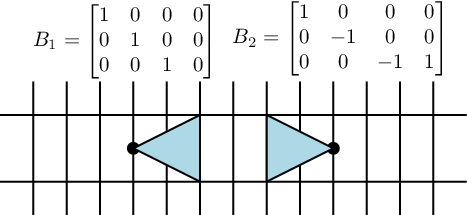}
  \caption{}
  \label{fig:trackopposingcams}
  \centering
  \includegraphics[width=.49\linewidth]{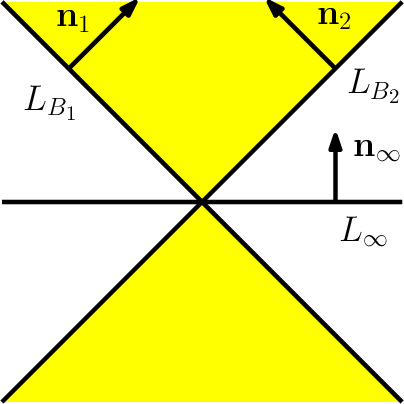}
\caption{}
\label{fig:smallinfiniteintersection}
\end{subfigure}

\caption{Two cameras on parallel tracks facing the same direction in \Cref{fig:tracksameface} and opposing directions in \Cref{fig:trackopposingcams}. The shaded area in \Cref{fig:biginfiniteintersection} and \Cref{fig:smallinfiniteintersection} represents the chiral domain of the arrangements  $\mathcal{A}$ and $\mathcal{B}$ around their intersection points with the hyperplane at infinity (that is after projective transformation of the corresponding \Cref{fig:tracksameface} and \Cref{fig:trackopposingcams}). Every infinite point is in the chiral domain of $\mathcal{A}$. The only infinite points in the chiral domain of $\mathcal{B}$ are those in the intersection $L_\infty \cap L_{B_1} \cap L_{B_2}$.}
\label{fig:two cam chiral domain}
\end{figure*}

This brings us to the following definition, that 
declares a point to be in front of a collection of cameras if and only if 
the point can be approached by a sequence of finite points that have 
chirality $1$ in each camera. 



\begin{definition}[Chiral Domain of $\mathcal{A}$ and Chirality]
 \label{def:chiral domain}
 Let $\mathcal{A}$ be an arrangement of finite cameras. Then the \emph{chiral domain} of $\mathcal{A}$, denoted as 
 $D_\mathcal{A}$, is the closure of the set
  \[
 \{\mathbf{q} \in \P^3 \,|\, \mathbf{q} \text{ finite},\,\ \forall A\in \mathcal{A}, \ \operatorname{depth}(\mathbf{q}; A) > 0 \}.
 \]
Moreover, a point $\mathbf{q} \in \P^3$ is said to have \emph{chirality 1 with respect to }$\mathcal{A}$, denoted as  $\chi(\mathbf{q} ; \mathcal{A}) = 1$, if and only if $\mathbf{q} \in D_\mathcal{A}$.
\end{definition}

The limit in the above definition is defined using the natural topology in $\P^3$ induced by the quotient map $\pi: \R^4 \minus \{0\} \to \P^3$ in which  $\pi(\mathbf{v}) = \pi(\mathbf{w})$ if and only if 
$\mathbf{v} \sim \mathbf{w}$. In this 
quotient topology, a set $U \subseteq \P^3$ is open if
and only if its preimage $\pi^{-1}(U)$ is open in $\R^4 \minus \{0\}$ in the Euclidean topology. 
Thus $\mathbf{q}\in\P^3$ is a limit point of  a sequence $\{\mathbf{q}_i\} \subseteq \P^3$ if and only if all 
open sets $\pi^{-1}(U)$ containing the line $\pi^{-1}(\mathbf{q})$ contains the line $\pi^{-1}(\mathbf{q}_i)$ for some $i$. The closure of a set $S \subseteq \mathbb{P}^3$ is the set of all limit points of sequences in $S$.

\begin{remark}
The chiral domain is nonempty if and only if it has nonempty interior. Indeed, if $D_\mathcal{A}$ is nonempty, there is a finite point $\mathbf{q} \in\P^3$ that has positive depth in all cameras of $\mathcal{A}$. Since depth depends continuously on the finite point, there is a neighborhood $U\subseteq \P^3$ of finite points with positive depth in all cameras. 
\end{remark}

We illustrate~\Cref{def:chiral domain} using the camera arrangements in~\Cref{fig:two cam chiral domain}. In~\Cref{fig:tracksameface} 
$\mathbf{n}_{1} = (0,0,1,0)$ and $\mathbf{n}_{2}=(0,0,1,-1)$ and hence the 
principal rays point in the same direction $(0,0,1)$ in $\R^3$.
The principal planes are parallel but distinct in $\R^3$, hence $L_{A_1}$ and 
$L_{A_2}$ intersect $L_\infty$ 
in a two-dimensional subspace 
of $\R^4$ (line in $\P^3$). \Cref{fig:biginfiniteintersection} shows this intersection after projecting from a common point of $L_{A_1}, L_{A_2}$ and $L_\infty$ in order to get a $2$-dimensional picture, where the projections of these three planes intersect in a point.
The vanishing point of the train tracks has chirality 1 with respect to both cameras as there is a sequence of finite points with chirality 1 in both cameras, that converges to it.  The chiral domain of the pair of cameras is shaded yellow and 
all of $L_\infty$ is visible in both cameras.

The cameras $B_1, B_2$ in 
\Cref{fig:trackopposingcams} have principal rays $\mathbf{n}_{B_1} = (0,0,1,0)$ and 
$\mathbf{n}_{B_2} = (0,0,-1,1)$. In $\R^3$, they are 
pointed in opposite directions.
The principal planes are again parallel but distinct in $\R^3$, hence $L_{B_1}$ and 
$L_{B_2}$ intersect $L_\infty$ in a projective line. The only part of $L_\infty$ in 
the chiral domain (shaded in yellow) is the intersection $L_\infty \cap L_{B_1} \cap L_{B_2}$ as seen in \Cref{fig:smallinfiniteintersection}. These examples show that when there are multiple cameras, 
the chirality of points on $L_{A_i}$ and $L_\infty$ need more care. 

Our next goal is to give an algebraic description of the chiral domain in terms of inequalities. The above example shows that $D_{\{A_1,\ldots,A_m\}}$ is not equal to the intersection of the chiral domains $D_{\{A_i\}}$ of the individual cameras. \Cref{fig:trackopposingcams} shows a counter example: The plane $L_\infty$ at infinity is in the chiral domain of both cameras $B_1$ and $B_2$ but only a part of the plane at infinity is visible in both cameras, that is in $D_{\{B_1,B_2\}}$. This also implies that, in order to obtain an inequality description of $D_\mathcal{A}$, it is not enough to simply relax the strict inequalities $\depth(\mathbf{q};A)>0$, which leads to the inequalities $(\mathbf{n}_\infty^\top \mathbf{q}) (\mathbf{n}_i^\top \mathbf{q}) \geq 0$. These inequalities cut out a set that is too big: again, all of them are satisfied for every point on the plane at infinity but $D_\mathcal{A}$ usually only contains a subset of it. The following inequality description of $D_\mathcal{A}$ takes this subtlety into account.

\begin{theorem} \label{thm:chiral set of an arrangement}
Let $\mathcal{A} = \{A_1,\ldots,A_m\}$ be an arrangement of finite cameras. If the chiral domain, $D_\mathcal{A}$, is nonempty, then it has the semi-algebraic description
\begin{equation} \label{eq:same sign conditions}
    D_\mathcal{A} = \left \{ 
    \mathbf{q} \in \P^3\ \vert \ 
    \forall i,j, \begin{array}{l} (\mathbf{n}_\infty^\top \mathbf{q}) (\mathbf{n}_i^\top \mathbf{q})   \geq 0,\\
    (\mathbf{n}_i^\top \mathbf{q})(\mathbf{n}_j^\top \mathbf{q})\geq 0
    \end{array}\right\},
\end{equation}
where $\mathbf{n}_i$ is the principal ray of $A_i$.
\end{theorem}

Note that the inequalities $(\mathbf{n}_\infty^\top \mathbf{q}) (\mathbf{n}_i^\top \mathbf{q})   \geq 0$ (and also $(\mathbf{n}_i^\top \mathbf{q})(\mathbf{n}_j^\top \mathbf{q})\geq 0$) are well-defined in  $\P^3$ because in each case, the polynomial $f$ on the left hand side has even degree (namely $2$) in $\mathbf{q}$, that is $f(\lambda \mathbf{q}) = \lambda^2 f(\mathbf{q})$ for every real scalar $\lambda$. Therefore, the sign of $f(\mathbf{q})$ is constant along the line spanned by $\mathbf{q}$ in $\R^4$. 

\begin{proof}[of \Cref{thm:chiral set of an arrangement}] 
If $D_\mathcal{A}$ is nonempty, then it has a nonempty interior. Let $S$ be the interior of $D_\mathcal{A}$, i.e., the set of finite points in $\P^3$ that have positive depth in all cameras in $\mathcal{A}$. Such points correspond to lines through the origin in $\R^4$, 
consisting of points $\mathbf{q}$ that have 
($\mathbf{n}_i^\top\mathbf{q} > 0$ and $\mathbf{n}_\infty^\top \mathbf{q} > 0$) or 
($\mathbf{n}_i^\top\mathbf{q} < 0$ and $\mathbf{n}_\infty^\top \mathbf{q} <  0$). 
Let $Q\subseteq \R^4$ be the polyhedral cone defined by the inequalities 
$\mathbf{n}_i^\top\mathbf{q} \geq 0$ and $\mathbf{n}_\infty^\top \mathbf{q}\geq 0$.
Then we see that $S = \P(\interior Q \cup - \interior Q)$ and hence 
the projectivization of $\interior Q$. This implies that  
$D_\mathcal{A}$ is the projectivization of 
$Q$ which can be defined by the quadratic inequalities $(\mathbf{n}_i^\top\mathbf{q})(\mathbf{n}_j^\top \mathbf{q})\geq 0$, where $\{i,j\}$ ranges over all $2$-element subsets of $\{1,2,\ldots,m,\infty\}$. \qed
\end{proof}

\begin{remark}
The inequality description of $D_\mathcal{A}$ 
in \eqref{eq:same sign conditions} is only valid when 
$D_\mathcal{A}$ is nonempty. Indeed, the set on the right hand side can be 
nonempty even if $D_\mathcal{A}$ is empty. 
For example, for the cameras in \Cref{fig:trackoppositeface}, $D_\mathcal{A} = \emptyset$ as there 
is no finite point with positive depth in both cameras. However, all points that lie on the line 
that is the intersection of $L_\infty$ with the (common) principal plane of the cameras 
satisfies the (non-strict) inequalities on the right hand side of \eqref{eq:same sign conditions} even though 
there is no point in $\P^3$ where the inequalities are satisfied strictly.
In general, the nonstrict inequalities admit all points that are on the principal planes of some cameras in $\mathcal{A}$ and have nonnegative depth in the others. 
\end{remark}

\begin{remark}
Note from the proof of \Cref{thm:chiral set of an arrangement} that 
$D_\mathcal{A}$ is the projectivization of the polyhedral cone $Q \subseteq \R^4$ 
defined by the linear inequalities $q_4\geq 0$ and $\mathbf{n}_i^\top \mathbf{q} \geq 0$ ($i=1,2,\ldots,m$). In other words, the lines in $\R^4$ corresponding to the points in 
$D_\mathcal{A}$ are exactly the lines through the origin in $Q \cup -Q$. Therefore, 
$D_\mathcal{A}$ is inherently a polyhedral set even though 
\Cref{thm:chiral set of an arrangement} describes $D_\mathcal{A}$ using quadratic inequalities.
\end{remark}

\begin{remark}
Specializing \Cref{thm:chiral set of an arrangement} to one camera, we get
$D_{\{A\}} = \{\mathbf{q}\in\P^3\colon (\mathbf{n}_A^\top \mathbf{q}) (\mathbf{n}_\infty^\top \mathbf{q}) \geq 0\}$, which implies that $\chi(\mathbf{q};\{A\}) = 1$ if $\mathbf{q} \in L_\infty \cup L_A$ for one camera $A$. This specialization matches our expectation for one camera that  
we explained earlier. 
\end{remark}

We now give a criterion for the non-emptiness of $D_\mathcal{A}$ in terms of linear programming.

\begin{theorem}\label{lem:chiralexistence}
Let $\mathcal{A} = \{A_1,A_2,\ldots,A_m\}$ be an arrangement of finite cameras.
Then $D_\mathcal{A} \neq \emptyset$ if and only if the row space of the $4 \times (m+1)$ matrix $N$ with columns ${\mathbf{n}}_1, \ldots, {\mathbf{n}}_m, {\mathbf n}_\infty$ intersects the positive orthant 
 $\R^{m+1}_{++}$.

In particular, for all arrangements of $m \le 3$ cameras such that $\mathbf{n}_\infty$ and the principal rays $\mathbf{n}_1, \dots, \mathbf{n}_m$ are linearly independent, $D_\mathcal{A} \neq \varnothing$.
\end{theorem}

\begin{proof}
The set $D_\mathcal{A} \neq \varnothing$ if and only if there is a finite point with positive depth in all 
cameras. Equivalently, if and only if there is a $\mathbf{q}\in\R^4$ such that $q_4 \neq 0$ where  $\mathbf{q}^\top N = [\mathbf{q}^\top \mathbf{n}_1, \ldots, \mathbf{q}^\top \mathbf{n}_m, q_4]$ lies in the positive or negative orthant. Thus $D_\mathcal{A} \neq \emptyset$ if and only if the row space of $N$ has an  intersection with $\R^{m+1}_{++}$. 

If $m\le 3$ and the columns of $N$ are linearly independent then $N$ has row  rank $m+1$, and the rows of $N$ span 
$\R^{m+1}$. So the rowspace of $N$ intersects $\R^{m+1}_{++}$. 
\qed\end{proof}

The following example shows that the three camera result in Theorem~\ref{lem:chiralexistence} is tight in the sense that when $m > 3$, the chiral domain may be empty.

\begin{example}
Consider an arrangement $\mathcal{A}$ with principal rays $\mathbf{n}_{1} = (-1,0,0,0)^\top, \mathbf{n}_2 = (1,-1,0,0)^\top, \mathbf{n}_3 = (0,1,-1,0)^\top,$ and $\mathbf{n}_4 =(0,0,1,-1)^\top$. The matrix 
\[
N = \begin{bmatrix} -1 & 1& 0 & 0 & 0 \\ 0 & -1 & 1 & 0 & 0 \\ 0&0&-1&1 &0 \\ 0&0&0&-1&1  \end{bmatrix} 
\]
has full rank. However, the row space of $N$ has empty intersection with $\R_{++}^5$, so  $D_\mathcal{A}$ is empty. 
\end{example}

\Cref{lem:chiralexistence} provides an efficient method for checking if  $D_\mathcal{A}$ is nonempty by checking the feasibility of a linear program (see \Cref{rem:LP check}), whose size scales linearly with the number of cameras.

\section{The Chiral Joint Image}
\label{sec:chiraljointimage}


Recall that world points are imaged in an arrangement of finite cameras $\mathcal{A}=\{A_1,\ldots,A_m\}$ via the 
rational map\footnote{Again, the broken arrow ($\dashrightarrow$) and the words ``rational map" refer to the fact that the domain of the map $\varphi_\mathcal{A}$ is not $\P^3$ but rather $\P^3\setminus\{\mathbf{c}_1,\ldots,\mathbf{c}_m\}$.}
\begin{align}
\varphi_\mathcal{A}: \left\{
\begin{array}{l}
\P^3 \dashrightarrow (\P^2)^m \\
\mathbf{q} \mapsto (A_1\mathbf{q},A_2\mathbf{q},\ldots,A_m\mathbf{q})
\end{array}
\right.
\end{align}
Triggs calls $\varphi_\mathcal{A}(\P^3) =: \mathcal{J}_\mathcal{A}$ the joint image~\cite{triggs1995geometry,triggs95} and Heyden-\AA str\"{o}m call it the {\em natural descriptor}~\cite{HA97}. In this section we will describe the chiral analog of the joint image, i.e. the set of images of points that lie in front of an arrangement of cameras.
\begin{definition}[Chiral Joint Image]
The \emph{chiral joint image} of a camera arrangement $\mathcal{A}$ is 
$\mathcal{X}_\mathcal{A} := \varphi_\mathcal{A}(D_\mathcal{A})$, the image of the chiral domain of $\mathcal{A}$ under $\varphi_\mathcal{A}$.
\end{definition}

Our goal will be to get an 
algebraic description of the chiral joint image $\mathcal{X}_\mathcal{A}$ and its 
Euclidean closure.
Since $\mathcal{X}_\mathcal{A}$ lies in the joint image $\mathcal{J}_\mathcal{A}$, we begin by 
looking at $\mathcal{J}_\mathcal{A}$ and its closures.

A {\em variety} in $\P^n$ 
is the set of solutions to a finite set of homogeneous polynomial equations. The varieties in $\P^n$ are the closed sets of the 
Zariski topology on $\P^n$.
The Zariski closure of a set $S \subset \P^n$, denoted by $\overline{S}^{Zar}$ is the smallest 
variety containing $S$.  

Closure in the Zariski topology is not just a mathematical nicety. In order to compute with a set, one needs a representation. By passing to its Zariski closure, we get the smallest algebraically representable set containing $\mathcal{J}_\mathcal{A}$.

Trager et al. \cite{THP15} refer to $\overline{\mathcal{J}}_\mathcal{A}^{Zar}$ 
in $(\P^2)^m$ as the {\em joint image variety of $\mathcal{A}$}. Recall that the {\em epipolar} and {\em trifocal} constraints cut out the joint image variety~\cite{idealsofthemultiviewvariety}. These are known as 
the {\em multiview constraints} on the image points.  Trager et al. also characterize the points added to $\mathcal{J}_\mathcal{A}$ 
by the closure operation, that is the difference between $\mathcal{J}_\mathcal{A}$ and $\overline{\mathcal{J}}_\mathcal{A}^{Zar}$. They do this using the following sets.

\begin{definition} \label{def:EA}
Given an arrangement of cameras  $\mathcal{A} = \{A_1,\ldots,A_m\}$, let  $E_j = \mathbf{e}_{1j} \times \hdots \P^2_j \hdots \times \mathbf{e}_{mj}$, 
where $\P^2_j$ represents a copy of $\P^2$ in the $j$th slot. Set $E_\mathcal{A} =  \bigcup_{j=1}^m E_j$, where $A_i \mathbf{c}_j = \mathbf{e}_{ij}$ are the epipoles of $\mathcal{A}$.
\end{definition}

\begin{theorem}[{\cite[Proposition~1]{THP15}}]\label{thm:thp15prop1}
Given an arrangement of cameras  $\mathcal{A} = \{A_1,\ldots,A_m\}$, with distinct camera centers, \[
\overline{\mathcal{J}}_\mathcal{A}^{Zar} = \mathcal{J}_\mathcal{A} \cup E_\mathcal{A}.
\]
\end{theorem}

While the Zariski topology is natural for algebraic sets, it is too coarse for semialgebraic sets. 
 For example, consider the set $C = \{(x,y) \in \mathbb{R}^2 | x^2 + y^2 = 1, x > 0, y > 0\}$, i.e., the unit circle restricted to the positive quadrant. When talking about its closure, one would want to talk about the set $\bar{C} = \{(x,y) \in \mathbb{R}^2 \,|\, x^2 + y^2 = 1, x \geq 0, y \geq 0\}$. However, the Zariski closure of $C$, $\overline{C}^{Zar} = \{(x,y) \in \mathbb{R}^2 | x^2 + y^2 = 1\}$ is the entire unit circle.  Luckily for us, as the following theorem shows, the Euclidean closure
 $\overline{\mathcal{J}}_\mathcal{A}$ 
 and the Zariski closure 
 $\overline{\mathcal{J}}_\mathcal{A}^{Zar}$ 
  of the joint image are the same~\footnote{Recall that the topology we use on $\P^n$ is induced by the Euclidean topology on $\R^{n+1} \minus \{0\}$. This induces a topology on the product of real projective spaces $(\P^2)^m$. Explicitly, a set $U_1 \times U_2 \times \ldots \times U_m \subseteq (\P^2)^m$ is open if and only if the sets $U_i \subseteq \P^2$ are all open sets.}.
 
\begin{theorem}
\label{thm:euclidean-is-zariski}
$\overline{\mathcal{J}}_\mathcal{A} =\overline{\mathcal{J}}_\mathcal{A}^{Zar} =   {\mathcal{J}}_\mathcal{A} \cup E_\mathcal{A}.$
\end{theorem}
\begin{proof}
Recall that $\mathbf{p}=(\mathbf{p}_{1},\mathbf{p}_2,\ldots,\mathbf{p}_m)\in E_j$ is of the form $(\mathbf{e}_{1j},\hdots \mathbf{p}_j, \hdots, \mathbf{e}_{mj})$ for some  $\mathbf{p}_j \in \P^2$. So all coordinates of $\mathbf{p}$ except $\mathbf{p}_j$ are the images of $\mathbf{c}_j = \begin{bmatrix}-G^{-1}_j \mathbf{t}_j\\ 1 \end{bmatrix}$. Consider now the curve 
\[
    \mathbf{v}(s) = \begin{pmatrix}
     s G_j^{-1}\mathbf{p}_j - G_j^{-1}\mathbf{t}_j\\
     1
    \end{pmatrix} = \begin{pmatrix}
     s G_j^{-1}\mathbf{p}_j \\ 0 \end{pmatrix} + \mathbf{c}_j
\]
as $s$ varies over $\R$. Then $\lim_{s \rightarrow 0} \varphi_\mathcal{A}(\mathbf{v}(s)) = \mathbf{p}$, since for $i \neq j$, 
$A_i \mathbf{v}(s) = s G_iG_j^{-1} \mathbf{p}_j + \mathbf{e}_{ij}$ and $A_j \mathbf{v}(s) = s \mathbf{p}_j \sim \mathbf{p}_j$.  
So $\mathbf{p} \in \overline{\mathcal{J}}_\mathcal{A}$, and hence $E_j \subseteq \overline{\mathcal{J}}_\mathcal{A}$. Therefore, by \Cref{thm:thp15prop1}, 
$\overline{\mathcal{J}}_\mathcal{A}^{Zar} \subseteq  \overline{\mathcal{J}}_\mathcal{A}$.
This means that 
\begin{align}
    \mathcal{J}_\mathcal{A} \subseteq \overline{\mathcal{J}}_\mathcal{A}^{Zar} \subseteq  \overline{\mathcal{J}}_\mathcal{A}
    \end{align}
and taking Euclidean closure throughout and 
noting that Zariski closed sets are also closed in the Euclidean topology, we get 
the first equality. The second equality is \Cref{thm:thp15prop1}.
\qed
\end{proof}

\begin{remark}
In general, over the complex numbers, the Euclidean closure and the Zariski closure of a \ constructible set are equal~\cite[Theorem I.10.1]{mumford1996red}. 
~\Cref{thm:euclidean-is-zariski}, however, is about real numbers, and the real Euclidean closure of a set is not always equal to the real part of the complex Euclidean closure. This is because, the Euclidean closure of the complex points recovers the real points and can sometimes produce isolated real points (real points that can be separated from the real points in the constructible set). For example, consider the map $(a:b) \mapsto (a^3+ab^2:a^2b+b^3:a^3)$ from $\mathbb{P}^1$ to $\mathbb{P}^2$. The image is a rational curve in $\mathbb{P}^2$ defined by the equation $y^2z = x^2(x-z)$ and has an isolated real singularity at $(0:0:1)$, which is not in the closure of the image of $\mathbb{P}^1$. Our proof of~\Cref{thm:euclidean-is-zariski} rests on the multi-linearity of the image formation map.
\end{remark}

We will now focus on giving a semialgebraic description of the chiral joint image $\mathcal{X}_\mathcal{A}$ and its Euclidean closure.
The essential inequalities enforcing chirality in the space of images are given in the following definition. 
\begin{definition}
\label{def:c_a}
Given an arrangement of finite cameras $A_i = \begin{bmatrix} G_i & \mathbf{t}_i\end{bmatrix}$, define $C_\mathcal{A}$ to be the set 
\begin{equation*}\label{eq:ca}
\left\{ \mathbf{p}:\,
\begin{array}{l}
      \det(G_i)p_{i3} (\mathbf{a}_i\times \mathbf{a}_j)^\top (\mathbf{b}_{ij} \times \mathbf{a}_j) \geq 0, \\
      \det(G_i)\det(G_j) p_{i3} p_{j3} (\mathbf{b}_{ij} \times \mathbf{a}_i)^\top(\mathbf{b}_{ij}\times \mathbf{a}_j) \geq 0
\end{array} \right\}
\end{equation*}
where $\mathbf{p} = (\mathbf{p}_1 ,\mathbf{p}_2,\ldots,\mathbf{p}_m)  \in(\P^2)^m$, $\mathbf{b}_{ij} = G_i^{-1}\mathbf{t}_i - G_j^{-1}\mathbf{t}_j$ is a direction of the baseline connecting the centers of cameras $A_i$ and $A_j$, and $\mathbf{a}_i = G_i^{-1} \mathbf{p}_i$.
\end{definition}

The inequalities describing $C_\mathcal{A}$ come from those describing the chiral domain $D_\mathcal{A}$. See proof of  \Cref{lem:worldpointCa2} in 
\Cref{sec:appendix}.


\begin{figure*}[h]
\begin{subfigure}{.49\textwidth}
  \centering
  \includegraphics[width=\linewidth]{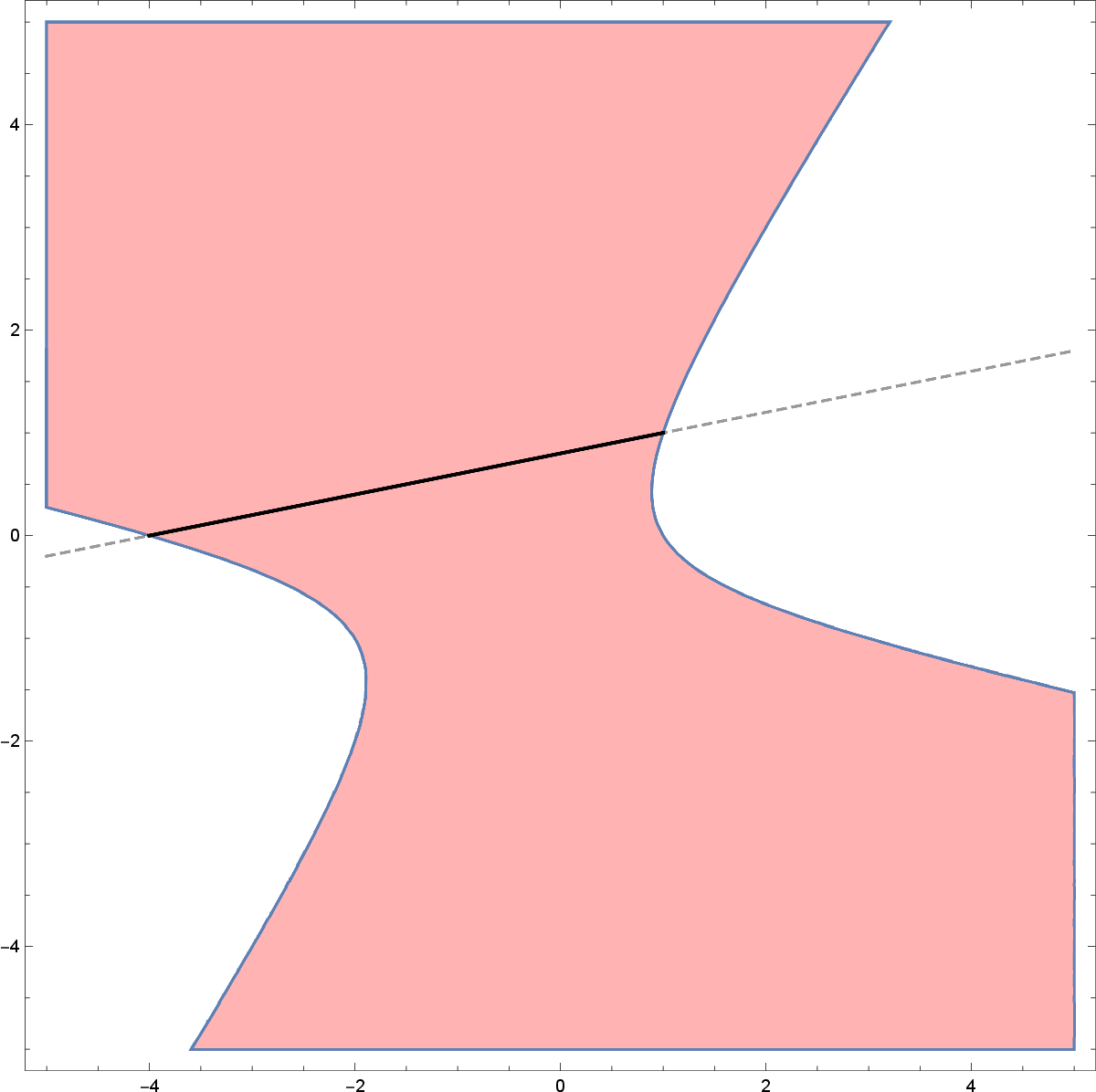}
\caption{The red region satisfies the inequality \\ \centering $(\mathbf{a}_1 \times \mathbf{a}_2)^\top (\mathbf{b}_{12} \times \mathbf{a}_2)\ge 0$.}
  \centering
 \includegraphics[width=\linewidth]{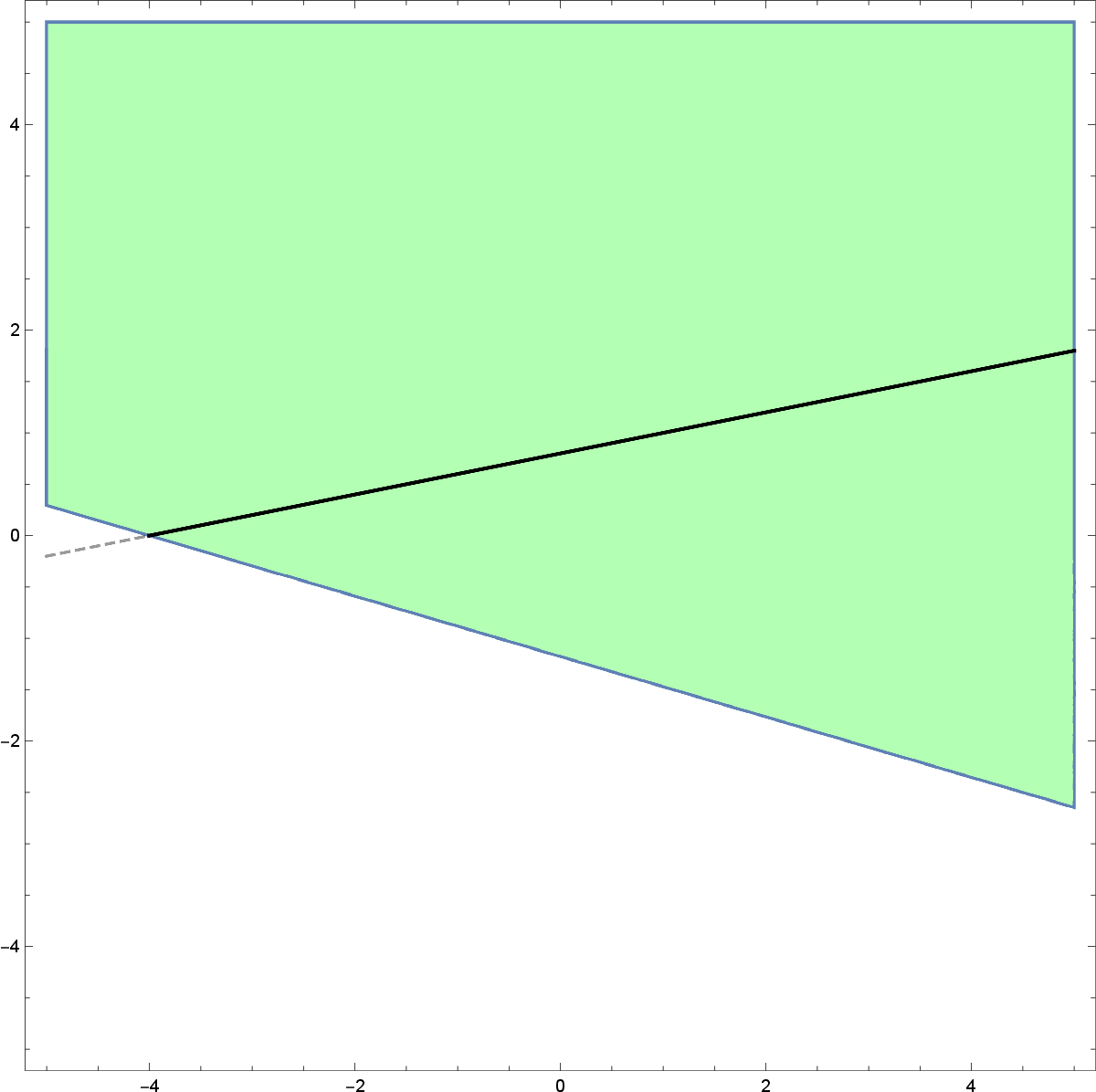} 
  \caption{The green region satisfies the inequality  \\ \centering $(\mathbf{a}_1 \times \mathbf{a}_2)^\top (\mathbf{b}_{12} \times \mathbf{a}_1)\ge 0$. }
\end{subfigure}
\begin{subfigure}{.49\textwidth}
  \centering
 \includegraphics[width=\linewidth]{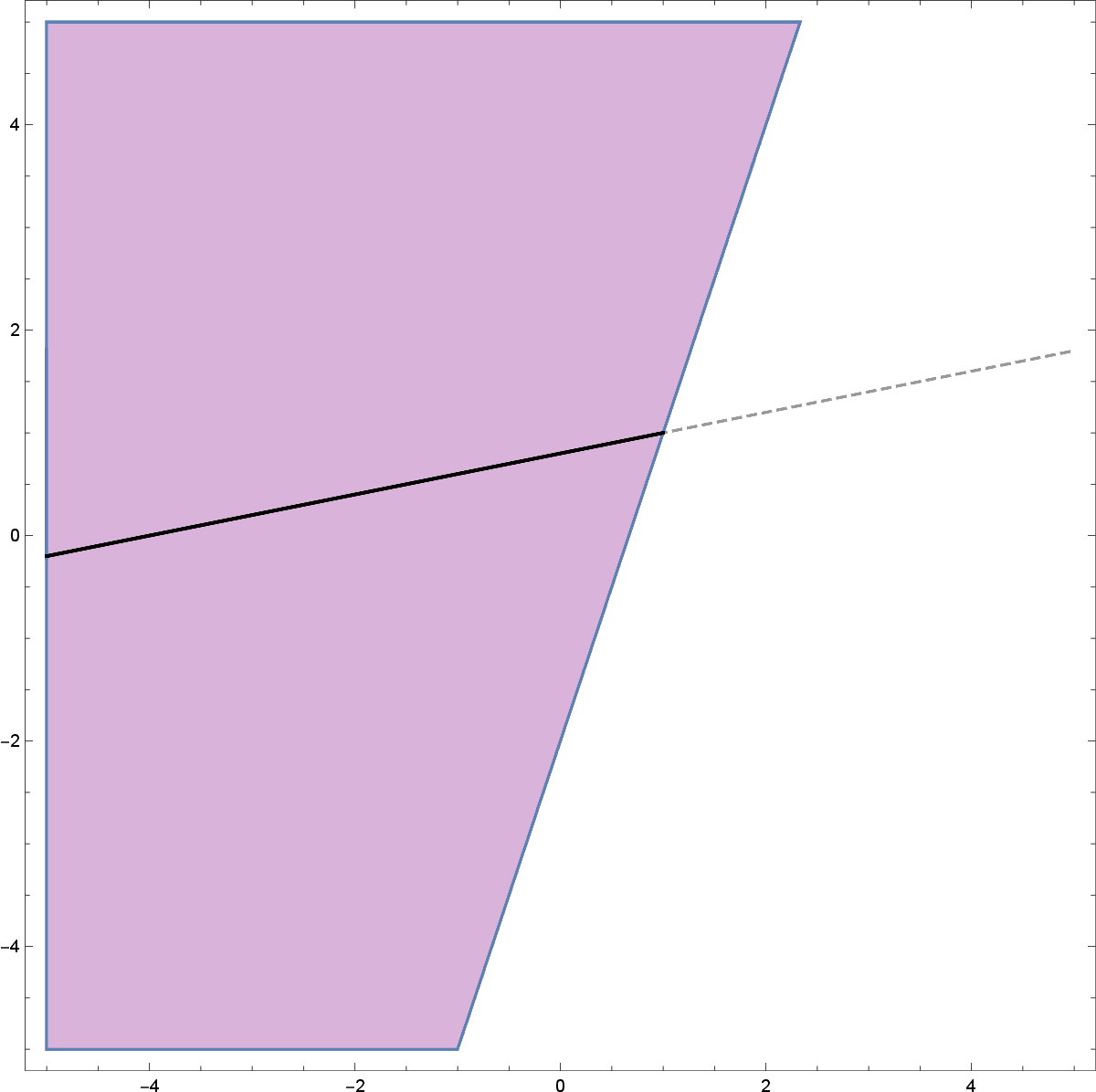}
\caption{The purple region satisfies the inequality  \\ \centering $(\mathbf{b}_{12} \times \mathbf{a}_1)^\top (\mathbf{b}_{12} \times \mathbf{a}_2)\ge 0$.}
  \centering
  \includegraphics[width=\linewidth]{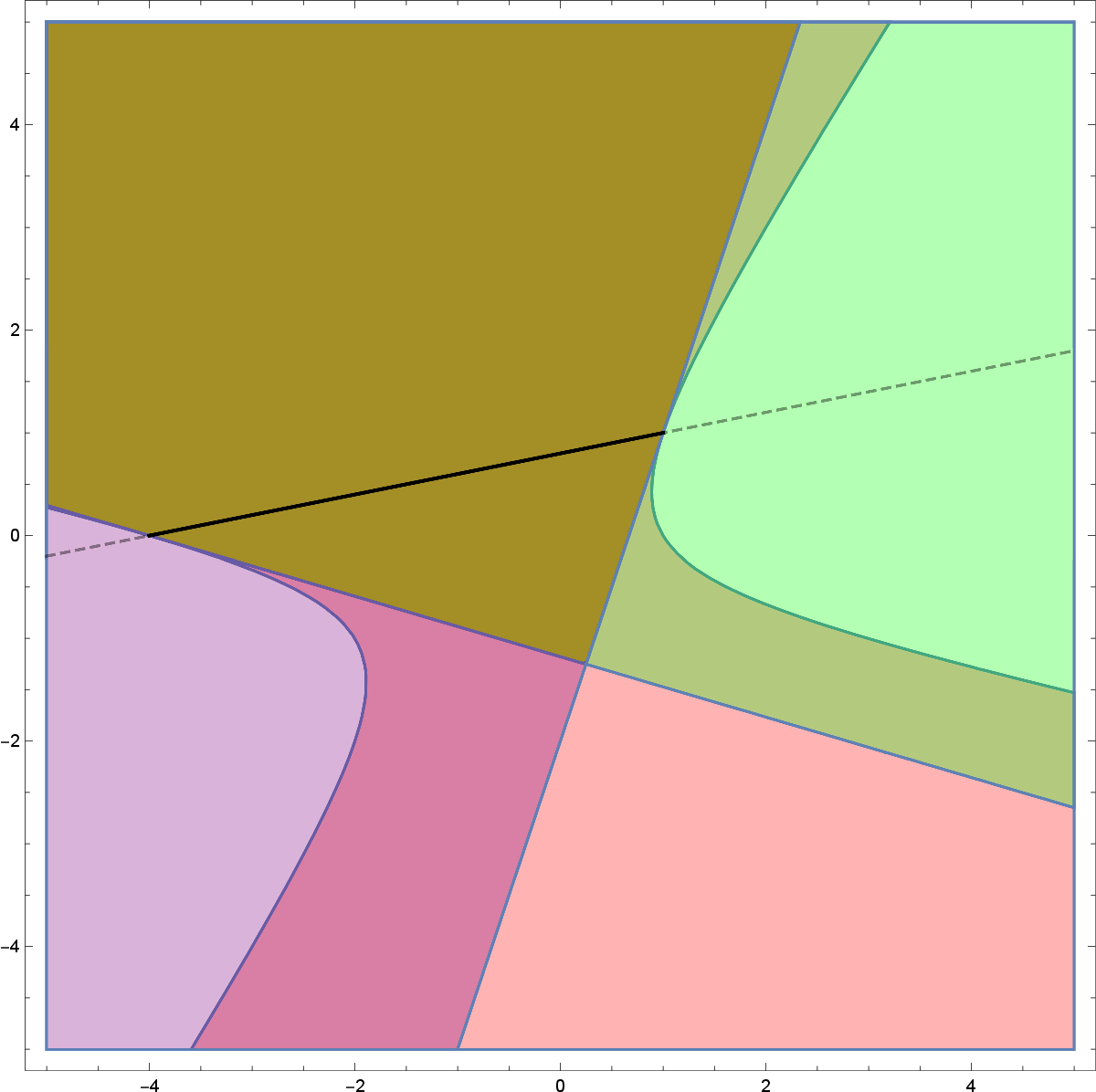} 
  \caption{Intersection of all colored regions satisfies all three inequalities.}
  \label{subfig:ca}
\end{subfigure}

\caption{ An illustration of how the various inequalities defining $C_\mathcal{A}$ (\Cref{def:c_a}) cut out the chiral joint image $\mathcal{X}_\mathcal{A}$ in \Cref{ex:cji}. In each figure, we fix $\mathbf{p}_1 = (-4,0,1)^\top$ and consider corresponding points in the second image. The dashed black line is in the epipolar line $\overline{\mathcal{J}}_\mathcal{A}$. The solid black line segments in figures (a), (b) and (c) is the intersection of the epipolar line with the (colored) region defined by one of the inequalities. The three inequalities are combined in figure (d) and their joint intersection with the epipolar line is shown in solid black. This is the closure of the chiral joint image $\overline{\mathcal{X}}_\mathcal{A}$ with $\mathbf{p}_1 = (-4,0,1)^\top$. }
\label{fig:cjiallregionsseperate}
\end{figure*}

\begin{figure*}[ht]
\begin{subfigure}{.3\textwidth}
  \centering
  \includegraphics[width=\linewidth]{CA1allregions.eps} 
  \caption{ $\mathbf{p}_1 = (-4,0,1)^\top$.}
\label{fig:ca minus four}
\end{subfigure}
\hspace{0.03\textwidth}
\begin{subfigure}{.3\textwidth}
  \centering
  \includegraphics[width=\linewidth]{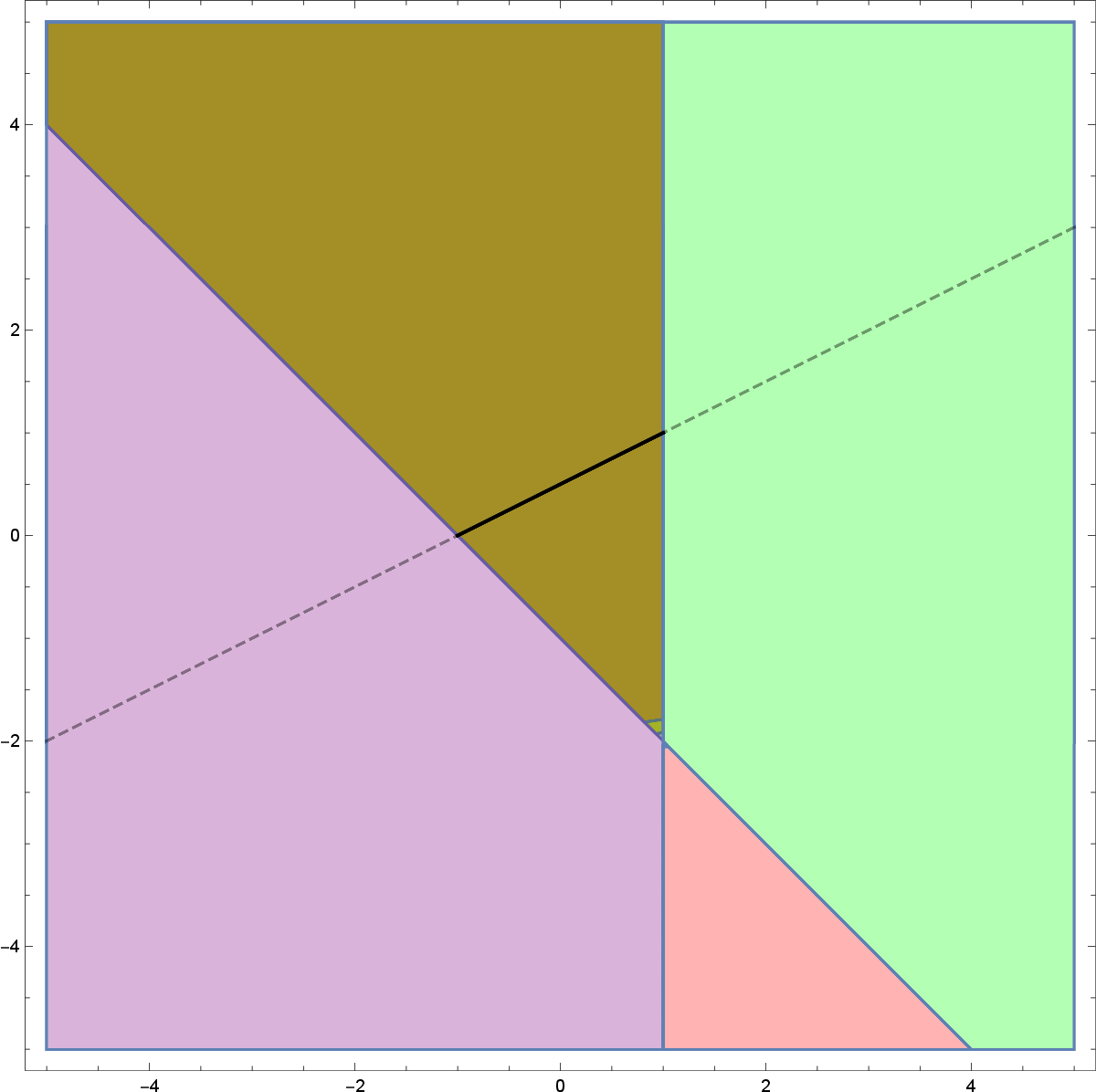}
\caption{ $\mathbf{p}_1 = (-1,0,1)^\top$.}
\end{subfigure}
\hspace{0.03\textwidth}
\begin{subfigure}{.3\textwidth}
  \centering
  \includegraphics[width=\linewidth]{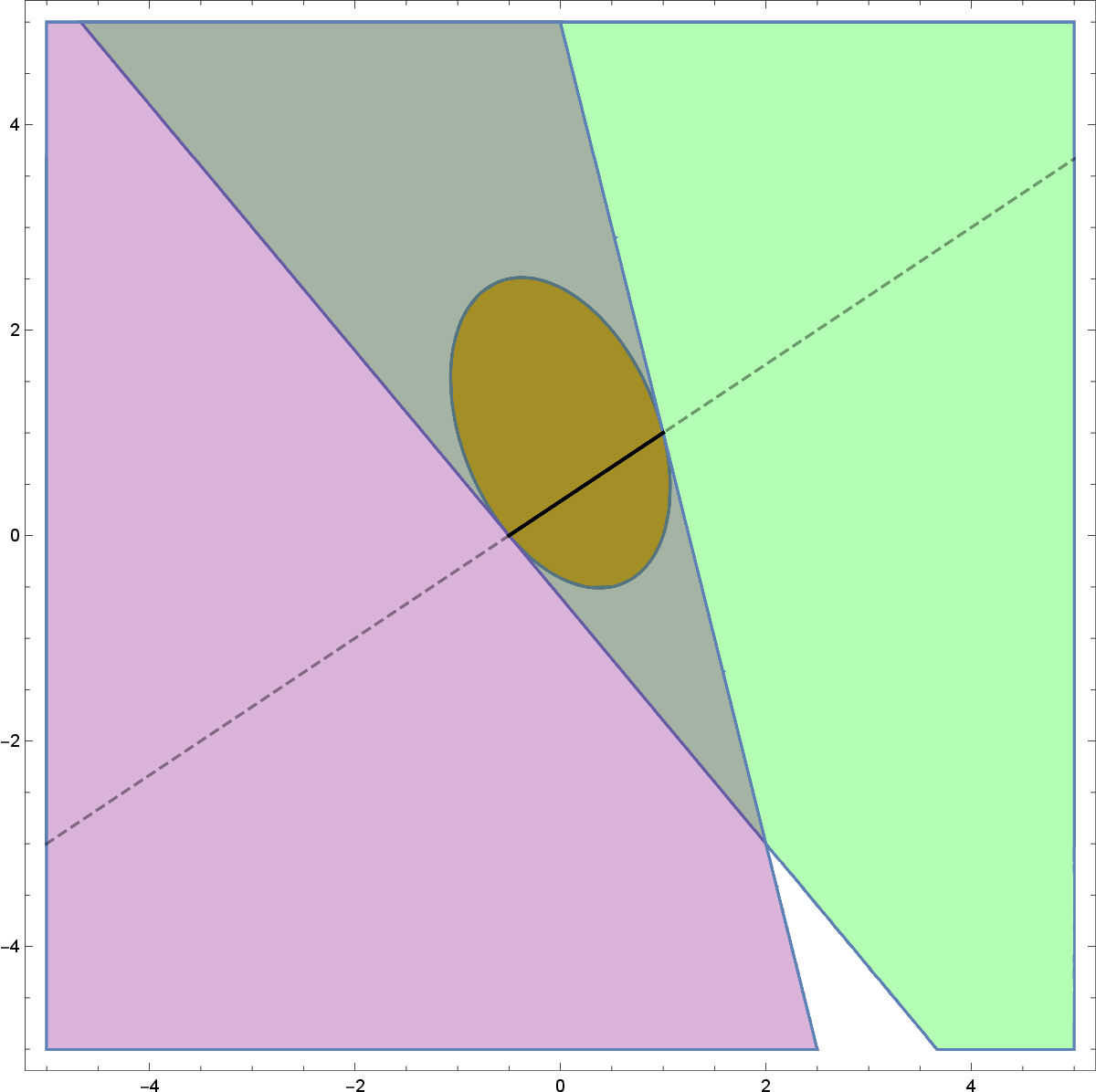} 
  \caption{ $\mathbf{p}_1 = (-.5,0,1)^\top$.}
\end{subfigure}
\hspace{0.04\textwidth}
\begin{subfigure}{.3\textwidth}
  \centering
  \includegraphics[width=\linewidth]{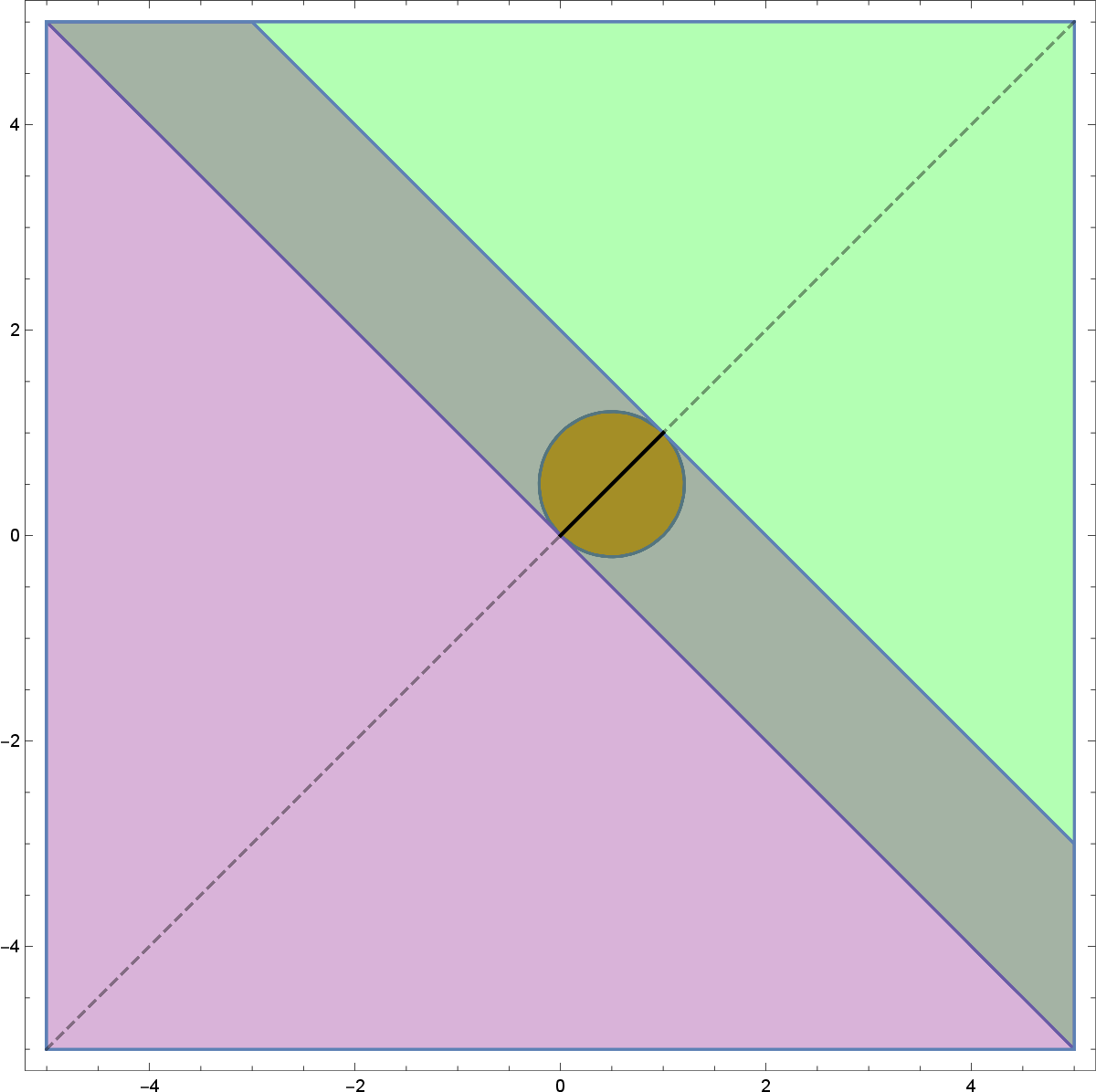}
\caption{ $\mathbf{p}_1 = (0,0,1)^\top$.}
\end{subfigure}
\hspace{0.03\textwidth}
\begin{subfigure}{.3\textwidth}
  \centering
  \includegraphics[width=\linewidth]{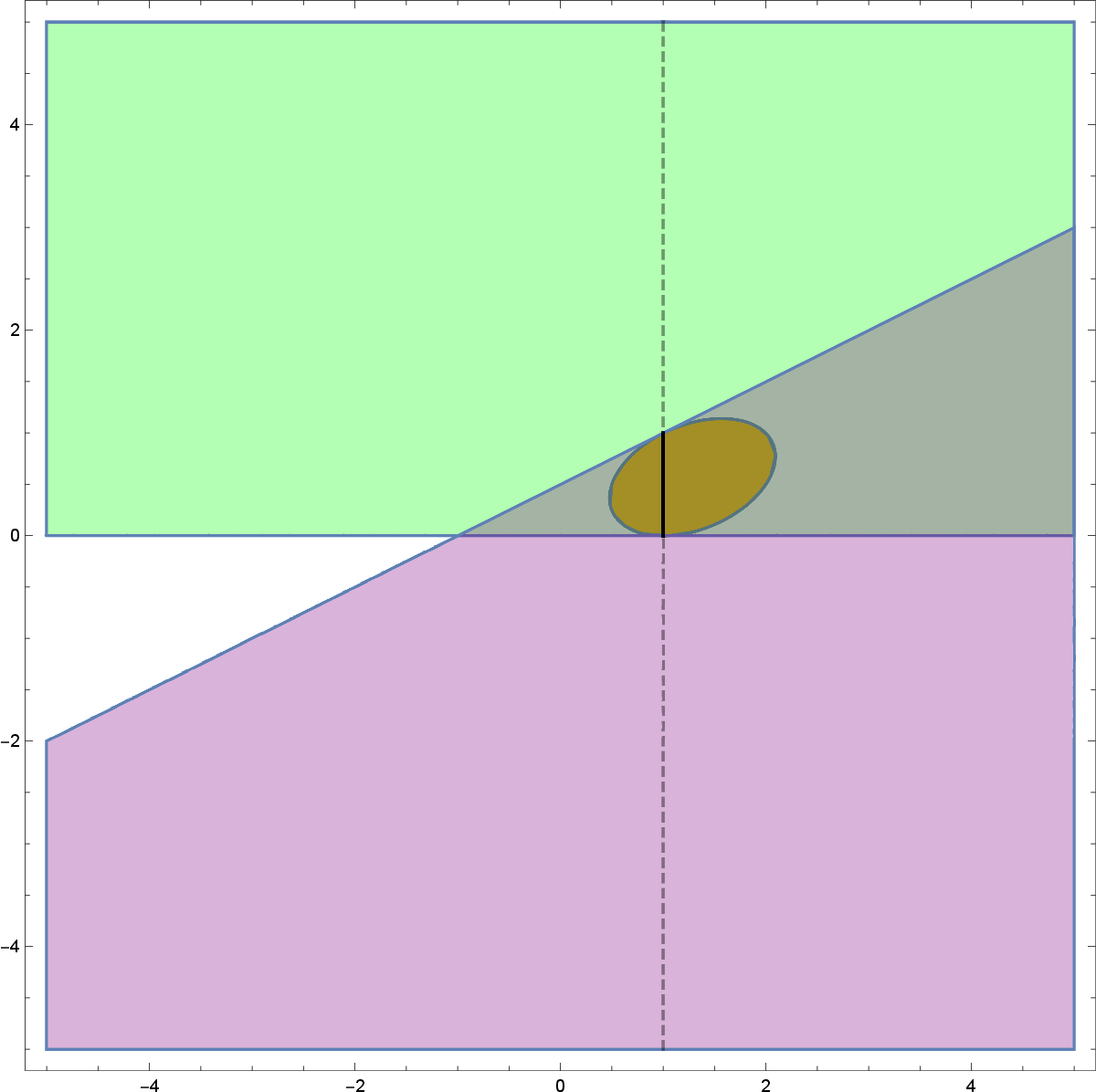}
\caption{ $\mathbf{p}_1 = (1,0,1)^\top$.}
\end{subfigure}
\hspace{0.03\textwidth}
\begin{subfigure}{.3\textwidth}
  \centering
  \includegraphics[width=\linewidth]{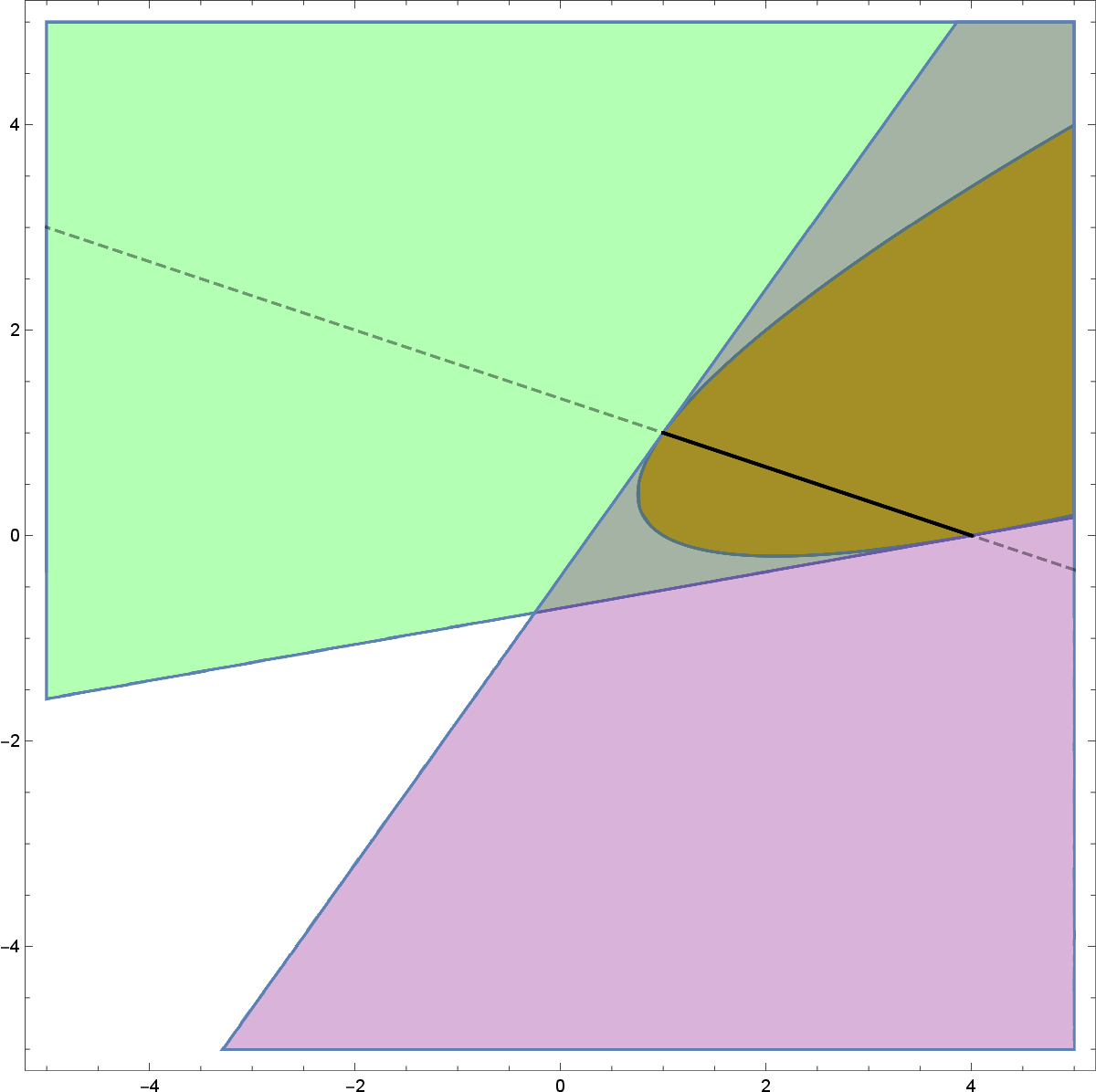}
\caption{ $\mathbf{p}_1 = (4,0,1)^\top$.}
\end{subfigure}
\hspace{0.04\textwidth}
\caption{
The chiral joint image for \Cref{ex:cji} as $\mathbf{p}_1$ varies along a line. The dashed black line is the epipolar line in ($\overline{\mathcal{J}}_\mathcal{A}$). The solid black line segment is in the closure of the chiral joint image ($\overline{\mathcal{X}}_\mathcal{A}$). See \Cref{fig:cjiallregionsseperate} for complete description of the regions. 
\label{fig:cjiallregions} }
\end{figure*}

Note that each inequality in $C_\mathcal{A}$ involves only two cameras in the arrangement. They are well-defined on $(\P^2)^m$ because every inequality has even degree in the coordinates on the $\P^2$-factors. In fact, the inequalities are all biquadratic, i.e.~of degree $(2,2)$. Moreover, the sign does not depend on the choice of the order of the cameras in the arrangement because this choice is implicit in $\mathbf{b}_{ij}$ and explicit in the terms $(\mathbf{a}_i\times \mathbf{a}_j)$ in the inequalities. So a relabeling of the cameras will not change the signs involved.

The following lemma (whose proof can be found in \Cref{sec:appendix}) gives the set theoretic relationship between the joint image, the chiral joint image and the set $C_\mathcal{A}$.  
\begin{lemma}\label{lem:worldpointCa2}
Let $\mathcal{A} = \{A_1,\ldots,A_m\}$ be an arrangement of finite cameras such that $D_\mathcal{A}$ is nonempty. If the centers of $\mathcal{A}$ are not collinear, then  
\begin{align}
\mathcal{X}_\mathcal{A} = \mathcal{J}_\mathcal{A}  \cap C_\mathcal{A}.
\end{align}
If the centers are collinear, then set 
$\mathbf{e} := (\mathbf{e}_1, \dots, \mathbf{e}_m)$ to be the image of the 
common baseline under $\varphi_\mathcal{A}$. 
Then 
\begin{align}
\mathcal{X}_\mathcal{A} \minus \{\mathbf{e}
\} = \left(\mathcal{J}_\mathcal{A}  \cap C_\mathcal{A}\right) \minus  \{\mathbf{e}\}
\end{align}
In both cases, $\overline{\mathcal{X}}_\mathcal{A} \subseteq C_\mathcal{A}$.
\end{lemma}

\Cref{lem:worldpointCa2}, while interesting, is not useful in practice, since it does not give us a way of algebraically representing $\mathcal{X}_\mathcal{A}$. This 
is because it is stated using $\mathcal{J}_\mathcal{A}$ which is not an algebraic set. A more useful description involves $\overline{\mathcal{J}}_\mathcal{A}$ which is algebraic. 

Recall that going from  $\mathcal{J}_\mathcal{A}$  to $\overline{\mathcal{J}}_\mathcal{A}$ brings in the set $E_\mathcal{A}$. For the chiral joint image the relevant part of $E_\mathcal{A}$  is 
    \[
     E^+_\mathcal{A} := \bigcup_{j\vert \mathbf{c}_j\in D_\mathcal{A}} E_j
    \]
    The set $E^+_\mathcal{A}$ can be divided into two parts as follows:
\begin{itemize}
        \item $E_\mathcal{A}^{++}$, the union of the sets $E_j$ such that $\mathbf{c}_j$ has positive depth in every camera $A_i$ with 
        $i \neq j$, and 
        \item $E^0_\mathcal{A}$, the union of all $E_j$ such that 
        $\mathbf{c}_j\in D_\mathcal{A}$ and the 
        depth of $\mathbf{c}_j$ is zero in some camera $A_i$ with $i\neq j$.
        \end{itemize}
Of these, the second set $E^0_\mathcal{A}$ causes the most technical issues as it may intersect (but not be contained in) $\overline{\mathcal{X}}_\mathcal{A}$, whereas it is always contained in $C_\mathcal{A}$. Armed with these definitions we state the main theorem of this section.

\begin{theorem}
\label{thm:chiraljointimageCa}
Let $\mathcal{A}$ be an arrangement of finite cameras with distinct centers. Further, assume that the chiral domain $D_\mathcal{A}$ is nonempty. Let $E^+_\mathcal{A}$ be the union of all $E_j$ such that $\mathbf{c}_j$ lies in $D_\mathcal{A}$.  
If the camera centers are in general position, i.e. $E^0_\mathcal{A} = \emptyset$, then
\begin{equation}
    \overline{\mathcal{J}}_\mathcal{A} \cap C_\mathcal{A}   =   \overline{\mathcal{X}}_\mathcal{A}. \label{eq:generic-chiral-joint-image}
\end{equation}
If $E^0_\mathcal{A} \neq \emptyset$  and the camera centers are not collinear, then
\begin{equation}
\overline{\mathcal{J}}_\mathcal{A} \cap C_\mathcal{A}  = \mathcal{X}_\mathcal{A}\cup E^+_\mathcal{A} =   \overline{\mathcal{X}}_\mathcal{A}\cup E^0_\mathcal{A}. \label{eq:full-chiral-joint-image}
\end{equation}
If the camera centers are collinear then
\begin{align}
\overline{\mathcal{J}}_\mathcal{A}\cap C_\mathcal{A} 
&= \mathcal{X}_\mathcal{A} \cup E^+_\mathcal{A} \cup \{\mathbf{e}\} = \overline{\mathcal{X}}_\mathcal{A} \cup E^0_\mathcal{A} \cup \{\mathbf{e}\}. \label{eq:collinear-chiral-joint-image} 
\end{align}
where $\mathbf{e} = ( \mathbf{e}_1, \ldots, 
\mathbf{e}_m)$ is the image of the common 
baseline under $\varphi_\mathcal{A}$.
\end{theorem}

By \Cref{thm:chiraljointimageCa}, the epipolar and trifocal constraints together with the inequalities defining $C_\mathcal{A}$ are the {\em chiral multiview constraints}. In the generic case they give an explicit semi-algebraic description of the closure of the chiral joint image (\cref{eq:generic-chiral-joint-image}).

Things get complicated when one or more of the cameras lie on the principal plane of another camera, i.e. the set $E^0_{\mathcal{A}}$ is non-empty. In this case the description becomes implicit (\cref{eq:full-chiral-joint-image}). Unfortunately the case of the non-empty $E^0_\mathcal{A}$ is not a pathology. Stereo cameras commonly involve two cameras whose centers are sitting on the common principal plane. More generally planar camera arrays have the same problem. Making the description explicit when $E^0_\mathcal{A} \neq \emptyset$ would require a more refined analysis because only parts of $E^0_\mathcal{A}$ are included in the closed set $\overline{\mathcal{X}_{\mathcal{ A}}}$. Such a refined analysis would be particularly relevant for specific cases like the stereo pair and the planar camera array.


Specializations of \Cref{thm:chiraljointimageCa} to Euclidean cameras are straightforward.

The proof of \Cref{thm:chiraljointimageCa} relies on the following additional lemmas, the proofs of which can be found in \Cref{sec:appendix}. 

\begin{lemma}\label{lem:eaplus}
Let $\mathcal{A} = \{A_1,\ldots,A_m\}$ be an arrangement of finite cameras. If the centers $\mathbf{c}_i$ are not collinear then $E_\mathcal{A} \cap C_\mathcal{A} = E^+_\mathcal{A}$.  Otherwise, 
 $E_\mathcal{A} \cap C_\mathcal{A} = E^+_\mathcal{A} \cup \{\mathbf{e}
 \}$.
\end{lemma}

\begin{lemma}
\label{lem:eaplusplus}
Let $\mathcal{A} = \{A_1,\ldots,A_m\}$ be an arrangement of finite cameras.
Then
$E_\mathcal{A}^{++}\subseteq  \overline{\mathcal{X}}_\mathcal{A}$.
\end{lemma}


\begin{proof}[of \Cref{thm:chiraljointimageCa}]
We will first prove \Cref{eq:full-chiral-joint-image} which is the case of noncollinear centers. 
By \Cref{thm:euclidean-is-zariski}, $\overline{\mathcal{J}_\mathcal{A}} = \mathcal{J}_\mathcal{A} \cup E_\mathcal{A}$. Therefore, 
$$\overline{\mathcal{J}}_\mathcal{A} \cap C_\mathcal{A} = (\mathcal{J}_\mathcal{A} \cap C_\mathcal{A}) 
\cup (E_\mathcal{A} \cap C_\mathcal{A}) = \mathcal{X}_\mathcal{A} \cup E^+_\mathcal{A}
$$
where the last equality follows from Lemmas \ref{lem:worldpointCa2} and \ref{lem:eaplus}. This proves the 
first equality in \Cref{eq:full-chiral-joint-image}. 

By \Cref{lem:eaplusplus}, 
$\overline{\mathcal{J}}_\mathcal{A} \cap C_\mathcal{A} \subseteq \overline{\mathcal{X}}_\mathcal{A} \cup E^0_\mathcal{A}$ since 
$$\overline{\mathcal{J}}_\mathcal{A} \cap C_\mathcal{A} = \mathcal{X}_\mathcal{A} \cup E^+_\mathcal{A} = \mathcal{X}_\mathcal{A} \cup E^{++}_\mathcal{A} \cup E^0_\mathcal{A} \subseteq \overline{\mathcal{X}}_\mathcal{A} \cup E^0_\mathcal{A}.$$
By \Cref{lem:worldpointCa2}, we get 
$\overline{\mathcal{X}}_\mathcal{A}\subseteq C_\mathcal{A}$.
Therefore, since $\mathcal{X}_\mathcal{A} \subseteq \mathcal{J}_\mathcal{A}$, it follows that $\overline{\mathcal{X}}_\mathcal{A}\subseteq \overline{\mathcal{J}}_\mathcal{A} \cap  C_\mathcal{A}$.
\Cref{lem:eaplus} shows that $E_\mathcal{A}^+$ (so in particular $E_\mathcal{A}^0$) is contained in $C_\mathcal{A}$, but since all of $E_\mathcal{A}$ is inside $\overline{\mathcal{J}}_\mathcal{A}$, we also have $E^0_\mathcal{A} \subseteq \overline{\mathcal{J}}_\mathcal{A} \cap  C_\mathcal{A}$. Therefore, 
$\overline{\mathcal{X}}_\mathcal{A} \cup E^0_\mathcal{A} \subseteq \overline{\mathcal{J}}_\mathcal{A} \cap  C_\mathcal{A}$. Putting both these containments together we get the second equality in 
\Cref{eq:full-chiral-joint-image}. 

The generic case, \Cref{eq:generic-chiral-joint-image} follows by observing that for cameras in general position $E^0_\mathcal{A} = \emptyset$.

The collinear case follows the same proof mechanics as above, utilizing the parts of  Lemmas \ref{lem:worldpointCa2}, \ref{lem:eaplus}, and \ref{lem:eaplusplus} that deal with collinear cameras.
\qed\end{proof}

We now illustrate the chiral joint image with the help of an example.

\begin{example}
\label{ex:cji}
Consider the pair of cameras $A_1 = \begin{bmatrix} I & \mathbf{0}\end{bmatrix}$, $A_2 = \begin{bmatrix} I & \mathbf{t}\end{bmatrix} $ where $\mathbf{t} = (1,1,1)^\top$.  We depict the set $C_\mathcal{A}$ and its intersection with 
$\overline{\mathcal{J}}_\mathcal{A}$ in \Cref{fig:cjiallregionsseperate} and \Cref{fig:cjiallregions}. Fixing $\mathbf{p}_1$, we plot the points on the corresponding epipolar line in the second image satisfying the inequalities defining $C_\mathcal{A}$.  Restricting to $p_{23} = 1$, we observe that one inequality is quadratic in the $\mathbf{p}_2$ factor while the other two inequalities are linear.
\end{example}


Observe how strong of a constraint chirality is by comparing the length of the epipolar line (dashed black) in each image to the length of the chiral joint image (solid black). 

This observation that chirality can be used to clip epipolar lines was first made by Werner and Pajdla ~\cite[Section 6]{WernerPajdla2001}. They argue geometrically that chirality may be used to restrict the search space for stereo-matching from a full epipolar line to a segment of the line. In practice, however, for an arbitrary pair of two cameras, enforcing chirality in their setting amounts to determining the feasibility of a linear program. Our methods reduce enforcing chirality to evaluating the closed form inequalities which define $C_\mathcal{A}$. In effect our derivation of the chiral joint image amounts to performing quantifier elimination on~\cite[Theorem 4]{WernerPajdla2001}.
Consequently, these inequalities offer a powerful new tool to constrain triangulation and multiview stereo matching algorithms.

\section{Chiral Reconstructions}
\label{sec:chiral-reconstructions}

In this section we 
show that the chiral domain recovers Hartley's seminal results on chirality. Our statements are sharper and more general because the chiral domain also includes points at infinity which were excluded from Hartley's framework. 
As a result we are also able to extend Hartley's results to Euclidean reconstructions.


In \cite{hartley1998chirality} and \cite[Chapter 21]{HartleyZisserman2004} Hartley shows that deciding whether a projective reconstruction of a collection of image correspondences 
\[\mathcal{P} := \{(\mathbf{p}_{1k},\dots,  \mathbf{p}_{mk}) \in (\R^2)^m, \,\,k=1, \ldots, n \}\]
can be made chiral by a homography of $\P^3$ reduces to solving a pair of linear programs. We recover this result using the chiral domain and present it in the language of polyhedral cones. 
This formulation leads to the notion of a {\em signed reconstruction} which we argue interpolates between projective and chiral reconstructions. For two cameras, Hartley proves that a projective reconstruction can be made chiral if and only if it can be signed. Our approach provides a concise algebraic proof of this result. While one direction of the proof is easy, the other direction is presented by Hartley only in \cite{hartley1998chirality} and is quite a bit more complicated than our argument. We then show via an example that the equivalence between signed and chiral reconstructions fails for more than two cameras and 
explain the reason for the gap.


\subsection{Projective Reconstructions}
\label{sec:projectiveexistence}

A \emph{projective reconstruction} of $\mathcal{P}$ is a pair $(\mathcal{A}, \mathcal{Q})$ consisting of  
an arrangement of $m$ finite cameras  $\mathcal{A} := \{A_1, \dots, A_m\}$ and a set of $n$ points $\mathcal{Q} := \{ \mathbf{q}_1, \dots, \mathbf{q}_n \} \subseteq \R^4 \minus \{\mathbf{0}\}$ such that $A_i\mathbf{q}_k = w_{ik} \widehat{\mathbf{p}}_{ik}$ for some scalars $w_{ik}$. 
By our definition of the chiral domain, it makes sense to say that the reconstructed scene is {\em in front of} the cameras in $\mathcal{A}$ if and only if the points $\mathbf{q}_k$ are in the chiral domain $D_\mathcal{A}$, which leads to the following definition:

\begin{definition}
\label{def:chiral reconstruction}
A \emph{chiral reconstruction} of $\mathcal{P}$ is a projective reconstruction $(\mathcal{A}, \mathcal{Q})$ of $\mathcal{P}$ such that $\mathbf{q}_k \in D_\mathcal{A}$ for all $k$.
\end{definition}

We call a reconstruction {\em finite} if the points in $\mathcal{Q}$ are finite. Recall that 
cameras are already required to be finite.
 In \cite{hartley1998chirality} and 
\cite{WernerPajdla2001} a {\em finite} projective reconstruction of $\mathcal{P}$ is called a {\em weak realization}, and a {\em finite} chiral reconstruction a {\em strong realization}. 
If there is a projective reconstruction of $\mathcal{P}$, there is always a finite one \cite{existenceprojectivereconstruction}. 


Consider now the question of when a given projective reconstruction $(\mathcal{A}, \mathcal{Q})$ of $\mathcal{P}$ can be transformed to a  {\em projectively equivalent} reconstruction $(\mathcal{A}H^{-1} , H \mathcal{Q})$ that is chiral, by a homography $H \in \GL_4$ of $\P^3$ where
\begin{align*}
\mathcal{A}H^{-1} : =&  \{A_1H^{-1}, \dots, A_m H^{-1}\}, \\
H\mathcal{Q}:=& \{H\mathbf{q}_1, \dots, H\mathbf{q}_n\}.
\end{align*}

The following lemma, parts of which appear in \cite[Section 5]{hartley1998chirality}, describes the effects of a homography $H$ on a reconstruction. A proof is provided in \Cref{sec:appendix}. Recall that, for the center $\mathbf{c}_A$ of a finite camera $A = \begin{bmatrix} G \,& \,{\mathbf t} \end{bmatrix}$, we choose the representative in $\R^4$ $$\mathbf{c}_A = \begin{bmatrix} -G^{-1} \mathbf{t} \\ 1 \end{bmatrix}.$$

\begin{lemma}
\label{lem:transformations}
Let $A = \begin{bmatrix} G \,& \,{\mathbf t} \end{bmatrix}$ be a finite camera with center $\mathbf{c}_A$.
Let $H \in \GL_4$ with last row $\mathbf{h}^\top$ and 
$\delta := \det(H^{-1})$. Then
\begin{enumerate}
    \item Under the homography $\mathbf{q} \mapsto H \mathbf{q}$, the plane $\mathbf{h}^\top \bf{q} = 0$ maps to the plane at infinity.
    \item The camera $AH^{-1}$ is finite if and only if $\mathbf{h}^\top \mathbf{c}_A \neq 0$. Its center then is $\mathbf{c}_{AH^{-1}} = \frac{1}{{\mathbf h}^\top \mathbf{c}_A} H\mathbf{c}_A$.
    \item The principal ray of $AH^{-1}$ is  
    $$\mathbf{n}_{AH^{-1}} = \delta ( \mathbf{h}^\top {\mathbf{c}}_A )  H^{-\top}\mathbf{n}_A,
    $$
    and so for all $\mathbf{q} \in \R^4$, we have 
    $$
    \mathbf{n}_{AH^{-1}}^\top(H\mathbf{q}) = \delta (\mathbf{h}^\top \mathbf{c}_A ) (\mathbf{n}_{A}^\top \mathbf{q}).
    $$
\end{enumerate}
\end{lemma}

Suppose we have a projective reconstruction $(\mathcal{A},\mathcal{Q})$ of $\mathcal{P}$.
Then recall that $\mathbf{n}_i^\top \mathbf{q}_k = \det(G_i) w_{ik}$ for all $i,k$. Since all $\widehat{\mathbf{p}}_{ik}$ have last coordinate $1$, $w_{ik} \neq 0$, and since all cameras are finite $\det(G_i) \neq 0$ for all $i$. Therefore, $\mathbf{n}_i^\top \mathbf{q}_k \neq 0$ for any $i,k$, or equivalently, no point $\mathbf{q}_k \in \mathcal{Q}$ lies on the principal plane of any camera $A_i$. Define $\si_{ik} = \sign(\mathbf{n}_i^\top \mathbf{q}_k) \in \{-1,1\}$ for all $i,k$. 
Using this notation we can state necessary and sufficient conditions for a projective reconstruction to become chiral under a homography.

\begin{theorem} \label{thm:conditions for mview chiral recon}
Suppose we have a projective reconstruction $(\mathcal{A} , \mathcal{Q})$ of $\mathcal{P}$ and  
$\si_{ik} = \sign(\mathbf{n}_i^\top \mathbf{q}_k)$. Then there is a $H \in \GL_4$ with last row $\mathbf{h}^\top$ such that 
$(\mathcal{A} H^{-1} , H\mathcal{Q})$ is a chiral reconstruction of $\mathcal{P}$ if and only if one of the following 
sets $\mathcal{S}_1$ or $\mathcal{S}_2$ is nonempty:
\begin{align}
  \label{eq:samesidesystem} \mathcal{S}_1 &= \left\{  \mathbf{h}\ | \ \forall i,j,k,\  \begin{array}{l} (\mathbf{h}^\top \mathbf{q}_k) (\mathbf{h}^\top \mathbf{c}_i)\si_{ik} \ge 0, \\
    (\mathbf{h}^\top \mathbf{c}_i) (\mathbf{h}^\top \mathbf{c}_j)\si_{ik}\si_{jk} > 0  \end{array}\right\} \\
  \label{eq:oppositesidesystem} \mathcal{S}_2 &= \left \{\mathbf{h}\ |\ \forall i,j,k,\ \begin{array}{l}  (\mathbf{h}^\top \mathbf{q}_k) (\mathbf{h}^\top \mathbf{c}_i)\si_{ik} \le 0, \\
(\mathbf{h}^\top \mathbf{c}_i) (\mathbf{h}^\top \mathbf{c}_j)\si_{ik}\si_{jk} > 0 \end{array}\right\}
 \end{align}
\end{theorem}
\begin{proof}
The reconstruction $(\mathcal{A}H^{-1}, H\mathcal{Q})$ of $\mathcal{P}$ is chiral 
if and only if for each $k$, $H\mathbf{q}_k$ lies in the chiral domain $D_{\mathcal{A}H^{-1}}$ of the camera arrangement $\mathcal{A}H^{-1}$. Therefore, 
from  \Cref{thm:chiral set of an arrangement}, Lemma \ref{lem:transformations},  and the 
requirement that cameras in the chiral reconstruction need to be finite, i.e., $\mathbf{h}^\top {\mathbf{c}}_i \neq 0$ for all $i$, $(\mathcal{A} H^{-1} , H\mathcal{Q})$ is chiral if and only if there exist $\mathbf{h}, \de$ such that for all $i,j,k$, 
\begin{eqnarray}
   &(\mathbf{n}_\infty^\top H \mathbf{q}_k)(\mathbf{n}_{A_i H^{-1}}^\top H \mathbf{q}_k) \\
   & = \de (\mathbf{h}^\top \mathbf{q}_k) (\mathbf{h}^\top {\mathbf{c}}_i) (\mathbf{n}_i^\top \mathbf{q}_k) \\
  \label{eq:niqkproduct}  & = \de (\mathbf{h}^\top \mathbf{q}_k) (\mathbf{h}^\top {\mathbf{c}}_i) \si_{ik} \ge 0
   \end{eqnarray}
   and 
   \begin{eqnarray}
     &(\mathbf{n}_{A_i H^{-1}}^\top H \mathbf{q}_k)(\mathbf{n}_{A_j H^{-1}}^\top H \mathbf{q}_k) \\ 
     & =  (\mathbf{h}^\top {\mathbf{c}}_i) (\mathbf{h}^\top {\mathbf{c}}_j)(\mathbf{n}_i^\top \mathbf{q}_k)(\mathbf{n}_j^\top \mathbf{q}_k) \\
     \label{eq:ninjqkproduct} & = (\mathbf{h}^\top {\mathbf{c}}_i) (\mathbf{h}^\top {\mathbf{c}}_j) \si_{ik} \si_{jk} > 0.   
 \end{eqnarray}
 Recall that we write $\mathbf{n}_i$ as shorthand for $\mathbf{n}_{A_i}$. 
 There are two sets  $\mathcal{S}_1$ and $\mathcal{S}_2$ to account for the sign of $\delta$. For 
 $\delta > 0$, the feasibility of \eqref{eq:niqkproduct} and \eqref{eq:ninjqkproduct} is  equivalent to 
$\mathcal{S}_1$ being nonempty. For $\delta < 0$, we get $\mathcal{S}_2$.  Lastly, any tuple $(\mathbf{h}, \de)$ can be completed to a $H \in \textup{GL}_4$ where $\mathbf{h}^\top$ is the last row of $H$ and $\det(H^{-1}) = \de$.
\qed
\end{proof}


\Cref{thm:conditions for mview chiral recon} inspires the algebraic notion of \emph{signing} a reconstruction. Indeed, if $\mathcal{S}_1$ or 
$\mathcal{S}_2$ is non-empty then the second set of inequalities in each set say that for each pair $i,j$,  the product 
$\si_{ik} \si_{jk}$ must be constant for all $k$. This will be guaranteed if for each camera we can choose one sign for $\si_{ik}$ for all $k$. 

\begin{definition}
\label{def:signed reconstruction}
A \emph{signed reconstruction} $(\mathcal{A}, \mathcal{Q}^s)$ of $\mathcal{P}$ is a projective reconstruction of $\mathcal{P}$ in which for each camera $i$, there exist constants $\si_i^s \in \{-1,1\}$ such that $\sign \left(\mathbf{n}_i^\top \mathbf{q}_k^s\right) = \si_i^s$ for all $k$. We say that a projective reconstruction $(\mathcal{A}, \mathcal{Q})$ can be signed if there exist $\mathbf{q}_k^s\in\R^4$ such that $\mathbf{q}_k^s \sim \mathbf{q}_k$ in $\P^3$ and $(\mathcal{A}, \mathcal{Q}^s)$ is a signed reconstruction.
\end{definition}

Note that signing a projective reconstruction $(\mathcal{A}, \mathcal{Q})$ changes it to $(\mathcal{A}, \mathcal{Q}^s)$ which only 
 amounts to changing the sign of some world points. It does not affect the cameras or chirality of the world points in these cameras. Geometrically, signing a reconstruction puts all chosen representatives of the world points in the same half space in $\R^4$, of the principal plane of each camera.

The following result shows that being able to sign a projective reconstruction is necessary to transform it to a chiral reconstruction. In this sense the signed reconstructions of $\mathcal{P}$ sit in between the projective reconstructions and chiral reconstructions of $\mathcal{P}$.

\begin{lemma}
\label{lem:signed reconstruction}
Suppose that a projective reconstruction $(\mathcal{A}, \mathcal{Q})$ of $\mathcal{P}$ is projectively equivalent to a chiral reconstruction of $\mathcal{P}$. Then for each pair $i,j$, the product $\si_{ik}\si_{jk}$ is constant for all $k$, and $(\mathcal{A}, \mathcal{Q})$ can be signed. 
\end{lemma}

\begin{proof}
We saw that the projective reconstruction $(\mathcal{A}, \mathcal{Q})$ can be made 
chiral only if either $\mathcal{S}_1$ or $\mathcal{S}_2$ is non-empty which happens 
only if for each pair $i,j$, the product $\sigma_{ik}\sigma_{jk}$ is constant for every $k$. 
In this case, we show that $(\mathcal{A}, \mathcal{Q})$ can be signed. For each $k$, define $\mathbf{q}_k^s := \mathbf{q}_k$ if $\si_{1k} = 1 $ or $\mathbf{q}_k^s := -\mathbf{q}_k$ if $\si_{1k} = -1$. By construction, $\si_{1k}^s := \sign( \mathbf{n}_1^\top \mathbf{q}_k^s ) =  1$ for all $k$. After this change, we still have $(\si_{1k}^s \si_{ik}^s)$ is constant for all $k$ since $\mathbf{q}_k$ appears quadratically in this expression. 
Then it follows that for each $i$, $\si_{ik}^s$ is constant for all $k$, and $(\mathcal{A}, \mathcal{Q}^s)$ is a signed reconstruction of $\mathcal{P}$.
\qed\end{proof}

Our concept of signing a reconstruction is 
equivalent to Hartley's Algorithm 21.1 (ii)  \cite{HartleyZisserman2004}. Signed reconstructions are closely related to {\em oriented projective reconstructions} in papers that model chirality using oriented projective geometry. As we do not discuss the oriented projective setup in this paper, we refer the interested reader to \cite{werner2003combinatorial}, \cite{werner2003constraint}, \cite{werner2001oriented} for details.

By \Cref{lem:signed reconstruction}, being able to sign a reconstruction is a necessary step in transforming to a chiral reconstruction.
In what follows we omit the superscript $s$ on a 
signed reconstruction, i.e., if we say that a projective 
reconstruction $(\mathcal{A},\mathcal{Q})$ is signed, then we mean that $\mathcal{Q} = \mathcal{Q}^s$ and $\sigma_i = \sigma_i^s$.

 We now rephrase the necessary and sufficient conditions for when a projective reconstruction can be made chiral in the language of polyhdral cones and their duals. Recall that 
 $K_\mathcal{Q} = \cone\{ \mathbf{q}_1 , \dots, \mathbf{q}_n\}. $ We define 
 $$K_{-\mathcal{Q}} = \cone\{ -\mathbf{q}_1 , \dots, -\mathbf{q}_n\}.$$ 
 Similarly, $K_{\si C}$ and $K_{-\si C}$ are the cones generated by 
 $\sigma C := \{ \si_1 {\mathbf{c}}_1, \dots, \si_m {\mathbf{c}}_m\}$ and 
 $-\si C$.
 
\begin{theorem}
\label{thm:cone conditions for chiral recon}
Given a signed reconstruction $(\mathcal{A} , \mathcal{Q})$ of $\mathcal{P}$, there 
is a $H \in \GL_4$ such that $(\mathcal{A} H^{-1} , H\mathcal{Q})$ is a chiral reconstruction if and only if
\begin{align}
   \label{eq:cone conditions for chiral recon} K_\mathcal{Q}^\ast \cap ( \interior K_{\si C}^\ast \cup \interior K_{-\si C}^\ast ) \neq \{\mathbf{0}\}
\end{align}
where  
$\interior K_{\si C}^\ast$ is the interior of the dual cone of 
$K_{\si C}$, and 
$K_\mathcal{Q}^\ast$ is the dual cone of $K_{\mathcal{Q}}$.
\end{theorem}

\begin{proof}
Since $(\mathcal{A} , \mathcal{Q})$ 
is a signed reconstruction, we may substitute the constants $\si_i$ for $\si_{ik}$.
Then $\mathcal{S}_1$ is the 
union of the cones 
$(K_\mathcal{Q}^\ast \cap \interior K_{\si C}^\ast)$ and
$(K_{-\mathcal{Q}}^\ast \cap \interior K_{-\si C}^\ast)$. Similarly, 
$\mathcal{S}_2$ is the union of   $(K_{-\mathcal{Q}}^\ast \cap \interior K_{\si C}^\ast)$ and $(K_\mathcal{Q}^\ast \cap \interior K_{-\si C}^\ast)$.
Since $K_\mathcal{Q}^\ast \cap \interior K_{\si C}^\ast \neq \{\mathbf{0}\}$ if and only if $K_{-\mathcal{Q}}^\ast \cap \interior K_{-\si C}^\ast \neq \{\mathbf{0}\}$, and $K_{-\mathcal{Q}}^\ast \cap \interior K_{\si C}^\ast \neq \{\mathbf{0}\}$ if and only if $K_\mathcal{Q}^\ast \cap \interior K_{-\si C}^\ast \neq \{\mathbf{0}\}$, finding a chiral reconstruction reduces to checking whether $K_\mathcal{Q}^\ast$ intersects one of the cones $\interior K_{\si C}^\ast $ or $\interior K_{-\si C}^\ast $.
\qed\end{proof}

The inequalities presented by the cone conditions in Theorem~\ref{thm:cone conditions for chiral recon} are essentially Hartley's chiral inequalities \cite[Equation~21.5]{HartleyZisserman2004}. 
In the rest of this section we use this 
interpretation of the chiral inequalities to recover and 
expand on Hartley's results on chirality.

\subsubsection{Quasi-affine transformations}

We first interpret \Cref{eq:samesidesystem} and \Cref{eq:oppositesidesystem} geometrically. \Cref{lem:transformations} shows that the effect of a homography on chirality is determined by the hyperplane it sends to infinity and its position relative to the camera centers and world points. The last row of a $4\times 4$ matrix $H$ representing a homography is the normal vector $\mathbf{h}$ of an oriented hyperplane in $\R^4$. 
By \Cref{lem:signed reconstruction}, a projective 
reconstruction of $\mathcal{P}$ can be made chiral by a homography only if 
it can be signed. Therefore, we may assume without loss of generality 
that we are starting with a signed reconstruction.

Given a signed reconstruction,
the second conditions of \Cref{eq:samesidesystem}, $(\mathbf{h}^\top \mathbf{c}_i) (\mathbf{h}^\top \mathbf{c}_j)\si_{i}\si_{j} > 0$, express that the camera centers should all be in the same (open) half-space given by $\mathbf{h}$.
The first conditions of \Cref{eq:samesidesystem}, $(\mathbf{h}^\top \mathbf{q}_k) (\mathbf{h}^\top \mathbf{c}_i)\si_{i} \ge 0$, say that the (possibly resigned by $\sigma_i$) world points, and camera centers, also lie in the same half spaces of this oriented hyperplane (or, in case of equality, combined with the second conditions, the world point lies on the hyperplane). 
Hence, these conditions encode a linear separation condition on the given points in $\R^4$, which can be checked via linear programming. The geometric interpretation of the inequalities in \Cref{eq:oppositesidesystem} is analogous: the (possibly resigned) world points lie in the opposite closed half-space defined by 
$\mathbf{h}$ to the camera centers.

Hartley presents this geometry in terms of {\em quasi-affine transformations} in  \cite{hartley1998chirality} and \cite[Chapter 21]{HartleyZisserman2004}. In \cite[Definition 21.3]{HartleyZisserman2004}, a homography $H$ is said to be quasi-affine with respect to a set $\mathcal{X} \subseteq \R^4$, with elements having last coordinate $1$, if no point in the convex hull of $\mathcal{X}$ is sent to infinity by $H$. We observe that this is equivalent to saying that $\mathbf{h}$, the last row of $H$, lies in $\interior K_\mathcal{X}^\ast$ or $\interior K_{-\mathcal{X}}^\ast$. To accommodate infinite points, we make a more general definition of a quasi-affine transformation.

\begin{definition}
\label{def:quasiaffine}
A linear map $H \in \textup{GL}_4$ is \emph{quasi-affine} with respect to $\mathcal{X} \subseteq \R^4$ if the last row $\mathbf{h}$ of $H$ lies in $K_\mathcal{X}^\ast \cup K_\mathcal{-X}^\ast$. Further, $H$ is \emph{strictly quasi-affine} with respect to $\mathcal{X}$ if $\mathbf{h} \in \interior K_\mathcal{X}^\ast \cup \interior K_{-\mathcal{X}}^\ast$.
\end{definition}

 Geometrically, $H$ is quasi-affine with respect to $\mathcal{X}$ if $H \mathcal{X}$ lies in one of the closed halfspaces defined by the  
hyperplane $\mathbf{h}^\perp = \{ \mathbf{x} \in\R^4\colon \mathbf{h}^\top \mathbf{x} = 0\}$, which is the plane sent to infinity
by the homography $H$. If $H \mathcal{X}$ lies in a open halfspace of 
$\mathbf{h}^\perp$ (as in Hartley's setup) then $H$ is strictly quasi-affine with respect to $\mathcal{X}$.

Recall that in a signed reconstruction $(\mathcal{A}, \mathcal{Q})$ we have fixed the sign of the last coordinates of all $\mathbf{q}_k \in \mathcal{Q}\subseteq \R^4$ and of all $\sigma_i {\mathbf{c}}_i$, and 
all points in $\mathcal{Q}$ and $\sigma C$ are considered to be in $\R^4$.
This allows Theorem \ref{thm:cone conditions for chiral recon} to be interpreted in terms of quasi-affine transformations.

\begin{theorem}
\label{thm:quasi-affine}
Suppose $(\mathcal{A}, \mathcal{Q})$ is a signed reconstruction of $\mathcal{P}$. Then there exists a chiral reconstruction of $\mathcal{P}$ if and only if there is a homography $H$ that is quasi-affine with respect to $\mathcal{Q}$ and strictly quasi-affine with respect to $\si C$. 
\end{theorem}

\begin{proof}
The intersection $K_Q^*\cap \left( \interior(K_{\si C}^*) \cup \interior(K_{-\si C}^*)\right)$
is nonempty if and only if 
$$\left(K_Q^*\cup K_{-Q}^*\right)\cap \left( \interior(K_{\si C}^*) \cup \interior(K_{-\si C}^*)\right)$$
is nonempty since ${\bf x} \in K_Q^*\cap \left( \interior(K_{\si C}^*) \cup \interior(K_{-\si C}^*)\right)$ if and only if 
$-{\bf x} \in K_{-Q}^*\cap \left( \interior(K_{\si C}^*) \cup \interior(K_{-\si C}^*)\right)$. The statement of the theorem is therefore equivalent to \Cref{thm:cone conditions for chiral recon}, by 
\Cref{def:quasiaffine}.
\qed\end{proof}

In Hartley's language, a ``quasi-affine reconstruction" is one which differs from a true scene by a quasi-affine transformation with respect to only the scene points. This is a weaker notion than a chiral reconstruction as Hartley points out (\cite[Section 8.1]{hartley1998chirality} and \cite[Section 21.]{HartleyZisserman2004}). By differentiating between strict and non-strict quasi-affine transformations we are able to state an if and only if theorem that connects quasi-affine transformations to chiral reconstructions.

Hartley's chiral inequalities are strict while our cone conditions in Theorem~\ref{thm:cone conditions for chiral recon} allow vectors $\mathbf{h}$ in the boundary of $K_\mathcal{Q}^\ast$. 
This is because the chiral domain is described by non-strict inequalities 
(\Cref{thm:chiral set of an arrangement}) which in turn came from extending the definition of chirality to all of $\mathbb{P}^3$. The reason to pass to $\interior K_{\si C}$ and $\interior K_{-\si C}$ was because of the need for finite cameras in a chiral reconstruction. In fact, we could have restricted to $\interior K_\mathcal{Q}^\ast$ in Theorem~\ref{thm:cone conditions for chiral recon} which would exactly give Hartley's chiral inequalities. This is because if the $\mathbf{h}$ produced in \Cref{thm:cone conditions for chiral recon} lies in the boundary of $K_\mathcal{Q}^\ast$, we may replace it by one in $\interior K_\mathcal{Q}^\ast$ by continuity. Geometrically this means that if $\mathcal{P}$ has a chiral reconstruction, then it has one in which all world points are finite and do not lie on any principal planes. 

Hartley's work was done with the aim of upgrading a two view projective reconstruction to a metric reconstruction. In follow up work, Nist{\'e}r addresses this question for multiple views \cite{nister2004untwisting}. He does this by transforming the projective reconstruction into one which is quasi-affine with respect to the camera centers.
As can be seen from \Cref{thm:quasi-affine} above, quasi-affineness with respect to the camera centers is a necessary condition for chirality. He does not enforce quasi-affineness with respect to the scene points, because they are often noisy and their chirality may change as part of the metric upgrade. Nist{\'e}r shows that enforcing the quasi-affineness on camera centers makes the iterative algorithm used to perform the subsequent metric upgrade easier and more reliable.

\subsubsection{Two-view chirality} 
We now recover Hartley's result that a two-view projective reconstruction can be made chiral if and only if it can be signed. Hartley remarks in \cite{hartley1998chirality} that the result 
does not extend to more than two cameras without further explanation. We use our conic tools to 
prove that the two-view result is tight and construct a counterexample with three cameras. We also explain the reason for the gap.

Suppose $(\{A_1, A_2\}, \mathcal{Q})$ is a two-view reconstruction of 
$\mathcal{P}$ 
such that $A_1$ and $A_2$ have distinct centers,  $A_1 \mathbf{q}_k = w_{1k} \widehat{\mathbf{p}}_{1k}$, and $A_2 \mathbf{q}_k = w_{2k} \widehat{\mathbf{p}}_{2k}$. Theorem 17 in \cite{hartley1998chirality} (also \cite[Theorem~1]{WernerPajdla2001}) gives a necessary and sufficient condition for when a two-view projective reconstruction can be transformed 
by a homography to a chiral reconstruction. We rederive this result  
in our language in \Cref{thm:Hartley paper thm} below. 
For the translation, recall that 
$\mathbf{n}_1^\top \mathbf{q}_k = \det (G_1) w_{1k} $, $ \mathbf{n}_2^\top \mathbf{q}_k = \det (G_2) w_{2k} $ and $\sigma_{ik} = \sign(\mathbf{n}_i^\top \mathbf{q}_k)$.
Therefore, the products $w_{1k} w_{2k}$ have the same sign for all $k$ if and only if  $(\mathbf{n}_1^\top \mathbf{q}_k)(\mathbf{n}_2^\top \mathbf{q}_k)$  have the same sign for all $k$, i.e., $\sigma_{1k}\sigma_{2k}$ is constant for all $k$.

The ``only if'' direction of \Cref{thm:Hartley paper thm} appears in both \cite[Theorem~17]{hartley1998chirality} and 
\cite[Theorem~21.7 (i)]{HartleyZisserman2004}, and the proof is straightforward. This argument 
is also the content of our \Cref{lem:signed reconstruction}.
The ``if'' direction appears in \cite{hartley1998chirality} with a rather complicated proof, 
and not in \cite{HartleyZisserman2004}.
We provide a short polyhedral proof of the ``if'' direction using 
\Cref{thm:cone conditions for chiral recon}. 
The conic formulation allows a simple proof via duality.

\begin{theorem}
\cite[Theorem 17]{hartley1998chirality}
\label{thm:Hartley paper thm}
A projective reconstruction $(\{A_1, A_2\}, \mathcal{Q})$ 
of $\mathcal{P}$ 
can be 
transformed by a homography $H$ to a chiral reconstruction if and only if $(\mathbf{n}_1^\top \mathbf{q}_k)(\mathbf{n}_2^\top \mathbf{q}_k)$ have the same sign for all $k$. 
\end{theorem}

\begin{proof}
Suppose $(\mathbf{n}_1^\top \mathbf{q}_k)({\mathbf{n}_2}^\top 
\mathbf{q}_k)$ have the same sign for all $k$. Then by 
\Cref{lem:signed reconstruction}, $\mathbf{n}_1^\top \mathbf{q}_k = \sigma_{1k} = \sigma_1$ and $\mathbf{n}_2^\top \mathbf{q}_k = \sigma_{2k} = \sigma_2$ for all $k$.

We first note that 
$\si_1 \mathbf{n}_1$ is a nonzero element of either $K_{\si C}^\ast$ or $K_{-\si C}^\ast$. We have that $(\si_1 \mathbf{n}_1)^\top (\si_1 {\mathbf{c}}_1) = 0$. If $\sign(\si_1 \mathbf{n}_1)^\top (\si_2 {\mathbf{c}}_2) = 1$ or $\sign(\si_1 \mathbf{n}_1)^\top (\si_2 {\mathbf{c}}_2) = 0,$ then $\si_1 \mathbf{n}_1 \in K_{\si C}^\ast$. Otherwise if $\sign(\si_1 \mathbf{n}_1)^\top (\si_2 {\mathbf{c}}_2) = -1$, then $\si_1 \mathbf{n}_1 \in (K_{-\si C}^\ast)$. 

Also, since the centers 
${\mathbf{c}}_1$ and ${\mathbf{c}}_2$ are distinct, 
$\sigma_1{\mathbf{c}}_1$ is not a scalar multiple of $\sigma_2 {\mathbf{c}}_2$, hence $K_{\sigma C}$ is a pointed cone.
 This implies that $K_{\sigma C}^\ast$ is full-dimensional and hence has an interior. The same is true for $K_{-\sigma C}^\ast$.

Without loss of generality suppose $\si_1 \mathbf{n}_1 \in  K_{\si C}^\ast$. Since 
$\sign(\mathbf{n}_1^\top\mathbf{q}_k) = \sigma_1$, we have that $\si_1\mathbf{n}_1^\top\mathbf{q}_k  >0$ for all $k$, and so $\si_1 \mathbf{n}_1 \in \interior K_\mathcal{Q}^\ast$. Let $U$ be a neighborhood of $\si_1 \mathbf{n}_1$ contained in $\interior K_\mathcal{Q}^\ast$. 
Since $\sigma_1 \mathbf{n}_1$ is also in $K_{\sigma C}^\ast$, there is some 
$\mathbf{h} \in U$ that lies in the $\interior K_\mathcal{Q}^\ast \cap \interior K_{\si C}^*$. This $\mathbf{h}$ 
is in $\mathcal{S}_1$, so by \Cref{thm:cone conditions for chiral recon}, $\mathcal{P}$ has a chiral reconstruction.
\qed\end{proof}

\Cref{thm:Hartley paper thm} shows that a chiral reconstruction exists if and only if $(\mathbf{n}_1^\top \mathbf{q}_k)(\mathbf{n}_2^\top \mathbf{q}_k)$ has the same sign for all $k$. \Cref{lem:signed reconstruction} shows that this is equivalent to being able to sign the reconstruction $(\{A_1,A_2\}, \mathcal{Q})$. Hence a two-view reconstruction can be made chiral if and only if it can be signed. Our notion of signing readily generalizes to multiple views. However, the following example shows that the ``if'' direction of \Cref{thm:Hartley paper thm} does not generalize to multiple views. In other words, it may not be possible to transform a signed reconstruction with three or more cameras into a chiral one.

\begin{example} \label{ex:signing not enough for 3 cameras}
Consider the reconstruction 
$$(\mathcal{A} = \{A_1,A_2,A_3\} ,\mathcal{Q} = (\mathbf{q}_1, \mathbf{q}_2))$$
where
\begin{align*}
A_1 = \begin{bmatrix} 0&0&-1&-1\\ 0&1&0&1 \\ 1&0&0&0 \end{bmatrix} A_2  = \begin{bmatrix} 1&0&0&1 \\0&0&-1 & 1\\ 0&1&0&0  \end{bmatrix} A_3 = \begin{bmatrix} 1&0&0&1\\ 0&1&0&-1 \\ 0&0&1&0  \end{bmatrix} 
\end{align*}
and $\mathbf{q}_1 = (1,1,2,-6)^\top$ and $\mathbf{q}_2 = (1,1,2,6)^\top$. 

The reconstruction is signed and $\si_1 = \si_2 = \si_3 = 1$ because $\si_{ik} = \sign(\mathbf{n}_i^\top \mathbf{q}_k) = 1 $ for all $i,k$. However, $(\mathcal{A}, \mathcal{Q})$ is not chiral because $\mathbf{q}_1$ is not in $D_\mathcal{A}$. Indeed, check that $({\bf n}_\infty^\top {\bf q}_1)({\bf n}_i^\top {\bf q}_1) < 0$ for 
all $i=1,2,3$.

We argue that $(\mathcal{A},\mathcal{Q})$ is not projectively equivalent to a chiral reconstruction using the conditions of \Cref{thm:cone conditions for chiral recon}. Consider the matrices $M^+ = \begin{bmatrix}
\mathbf{c}_1 & \mathbf{c}_2 & \mathbf{c}_3 & \mathbf{q}_1 & \mathbf{q}_2
\end{bmatrix}$ and $M^- = \begin{bmatrix}
-\mathbf{c}_1 & -\mathbf{c}_2 & -\mathbf{c}_3 & \mathbf{q}_1 & \mathbf{q}_2
\end{bmatrix}$, or explicitly
\begin{align*}
M^+ &= \begin{bmatrix} 0&1&-1&1&1 \\ -1&0&1&1&1 \\ -1&1&0&2&2 \\ 1&1&1&-6&6 \end{bmatrix}, \,\,\, M^- &= \begin{bmatrix}  0&-1&1&1&1 \\ 1&0&-1&1&1 \\ 1&-1&0&2&2 \\ -1&-1&-1&-6&6 \end{bmatrix}.
\end{align*}
Both $M^+$ and $M^-$ have a strictly positive kernel element. In particular, $M^+ \mathbf{v}^+ = 0$ and $M^-\mathbf{v}^- = 0$ where $\mathbf{v}^+ = (11,1,6,4,1)^\top  > 0$ and $\mathbf{v}^- =(1,11,6,1,4)^\top > 0 $. Existence of $\mathbf{v}^+$ and $\mathbf{v}^-$ shows that the linear systems 
\[
\{\mathbf{h} : \mathbf{h} \neq 0, \mathbf{h}^\top M^+  \ge 0 \} \, \text{ and }  \, \{\mathbf{h} : \mathbf{h} \neq 0, \mathbf{h}^\top M^-  \ge 0 \}
\]
are infeasible. Indeed, suppose there is a $\mathbf{h}\neq 0$ such that $\mathbf{h}^\top M^+ = \mathbf{y} \ge 0$. As $M^+$ has full rank, $\mathbf{y} \neq 0$. Since $\mathbf{v}^+$ is in the null space of $M^+$, it is orthogonal to the row space of $M^+$. It follows that $\mathbf{y}^\top \mathbf{v}^+ = 0$, but this is a contradiction because $\mathbf{v}^+$ is strictly positive and $\mathbf{y} \neq 0$. An analogous argument applies to the  linear system involving $M^-$.

Translated into cone language, this means $(K_\mathcal{Q}^\ast \cap \interior K_C^\ast) \cup (K_\mathcal{Q}^\ast \cap \interior K_{-C}^\ast) = \varnothing$. By \Cref{thm:cone conditions for chiral recon}, $(\mathcal{A}, \mathcal{Q})$ is not projectively equivalent to a chiral reconstruction. 
\end{example}

In general, if we have $m \geq 3$ cameras, then the hyperplanes with normals $\sigma_i \mathbf{c}_i$ partition $\R^4$ into $2^m$ (possibly empty) regions, each indexed by an element of $\{+,-\}^m$. It can be that $K_\mathcal{Q}^\ast$ lies entirely in a region of mixed signs forcing $K_\mathcal{Q}^\ast \cap (\interior K_{\sigma C}^\ast \cup \interior K_{-\sigma C}^\ast) = \emptyset$. 

For two cameras, this does not happen as we saw in the proof of \Cref{thm:Hartley paper thm}; $\sigma_1 \mathbf{n}_1$ is a non-zero element 
in $(K_\mathcal{Q}^\ast \cap K_{\sigma C}^\ast) \cup (K_\mathcal{Q}^\ast \cap K_{-\sigma C}^\ast)$. This relied crucially on the fact that $\sigma_1 \mathbf{n}_1$ is on the hyperplane with normal $\sigma_1 \mathbf{c}_1$ which is divided into two halfspaces by the hyperplane with normal $\sigma_2 \mathbf{c}_2$. Regardless of which half space $\sigma_1 \mathbf{n}_1$ lies in, it belongs to either $K_{\sigma C}^\ast$ or $K_{-\sigma C}^\ast$, i.e., it is automatically in either the $++$ or $--$ regions of hyperplanes with normal $\sigma_1 \mathbf{c}_1$ and $\sigma_2 \mathbf{c}_2$. This argument works for any number of cameras if $K_{\si N} := \cone(\si_1\mathbf{n}_1, \dots, \si_m\mathbf{n}_m)$ intersects 
$K_{\sigma C}^\ast$ or $K_{-\sigma C}^\ast$ because $K_{\si N }\subseteq K_\mathcal{Q}^\ast$ for any signed reconstruction. 
For $m > 2$ cameras, it can be that $K_{\sigma N}$, and even all of $K_\mathcal{Q}^\ast$, lies in a region of mixed signs of the hyperplane arrangement with oriented normals $\sigma_i \mathbf{c}_i$, as in Example~\ref{ex:signing not enough for 3 cameras}.

\subsection{Euclidean Reconstructions}
\label{sec:euclideanexistence}

In the previous section, we asked when a projective reconstruction can be transformed to a chiral reconstruction. We now ask the same question for a Euclidean reconstruction of $\mathcal{P}$, by which we mean a projective reconstruction $(\mathcal{A}, \mathcal{Q})$ in which each camera has the form $\begin{bmatrix} R & \mathbf{t}\end{bmatrix}$ where $R \in SO(3)$. 

Unlike for projective reconstructions, it is not true that if a Euclidean reconstruction exists, 
there is always one that is finite. However, this is not a problem since 
our definition of chiral reconstruction allows world points to be infinite 
thus generalizing the old notion of a strong realization.

Proposition 10 in \cite{hartley1998chirality} shows that we can assume $A_1 = \begin{bmatrix}  I & \mathbf{0} \end{bmatrix}$ by applying an appropriate similarity, without affecting chirality. Under this assumption, the following two theorems (whose proofs appear in \Cref{sec:appendix}) answer the above  question for $m=2$ and $m > 2$ views respectively.

\begin{theorem}
\label{thm:calibrated-2-views}
Let $(\{A_1 = \begin{bmatrix} I & \mathbf{0}\end{bmatrix} ,A_2 = \begin{bmatrix} R &\mathbf{t}\end{bmatrix} \}, \mathcal{Q})$ be a signed Euclidean reconstruction of $\mathcal{P}$ with distinct centers. There exists a 
chiral Euclidean reconstruction of $\mathcal{P}$ if and only if $\mathbf{n}_\infty \in K_\mathcal{Q}^\ast \cup K_{-\mathcal{Q}}^\ast$ or $\mathbf{r} := \begin{bmatrix} -\frac{2}{\|\mathbf{t}\|^2} R^\top \mathbf{t} \\1 \end{bmatrix} \in K_\mathcal{Q}^\ast \cup K_{-\mathcal{Q}}^\ast$. Equivalently, if exactly one of the following holds for all $\mathbf{q}_i$:
\[
q_{i4} \ge 0 \;\text{ or } \; q_{i4} \le 0 \; \text{ or } \; \mathbf{r}^\top \mathbf{q}_i \ge 0  \; \text{ or } \; \mathbf{r}^\top \mathbf{q}_i \le 0.
\]
\end{theorem}

\begin{theorem}
\label{thm:calibrated-m-views}
Let $(\mathcal{A}, \mathcal{Q})$ be a signed Euclidean reconstruction of $\mathcal{P}$ with $m>2$ cameras, distinct centers, and $A_1 = \begin{bmatrix} I& \mathbf{0}\end{bmatrix}$. There exists a chiral Euclidean reconstruction of $\mathcal{P}$ if and only if $\si_i = \si_j$ for all $1\le i<j \le m$ and either $q_{i4} \ge 0$ for all $i$ or $q_{i4} \le 0$ for all $i$.

\end{theorem}

These theorems are specializations of \Cref{thm:cone conditions for chiral recon}. Their proofs are based on the observation that restricting the cameras to be Euclidean restricts the class of homographies in \Cref{thm:cone conditions for chiral recon} to four ($m=2$) and two ($m > 2$) discrete choices respectively. The four choices for $m=2$ correspond to the well known twisted pair transformations and the two choices for $m>2$ correspond to reflection. 

\section{Summary} \label{sec:summary}
We introduce the chiral domain of an arrangement of cameras --- a multiview generalization of the definition of chirality that covers all of $\P^3$ --- 
and give a semialgebraic description of this set. 

We define the chiral joint image of a camera arrangement 
to be the image of the chiral domain in the cameras; 
it is the true image of the world in the cameras.
The chiral joint image lives naturally in the joint image variety of the camera 
arrangement, a classical quasi-projective 
variety in multiview geometry. We provide a complete semialgebraic description of the  chiral joint image.

The equations and inequalities describing the chiral joint image are the chiral analogs of the familiar multiview constraints. They lay the foundations for the development of a theory of chiral reconstruction. Our algebraic descriptions of the chiral domain and the chiral joint image can be used to enforce chirality when solving reconstruction or triangulation problems. Similarly, the chiral joint image can be used to constrain the region used for stereo matching.

The chiral domain framework also readily gives rise to quasi-affine transformations which are central to Hartley's work on chirality. This allows us to recover Hartley's chiral inequalities whose feasibility characterizes when a projective reconstruction can be made chiral by a homography.  Our approach provides a simple proof of the hard direction of Hartley's theorem that says that a two-view reconstruction can be made chiral by a homography if and only if the reconstruction satisfies a sign condition. We provide an example to show that such a sign condition does not suffice when there are more than two cameras. By extending the definition of chirality to all of $\mathbb{P}^3$ we are also able to extend Hartley's results to  Euclidean cameras.


\section{Technical Proofs} \label{sec:appendix}

This section contains the proofs of statements not proved in the main body of the paper for narrative clarity. The numbering of theorems and lemmas matches those in the main paper and they are presented here in the order in which they appear in the main paper.
In some cases, these proofs rely on additional lemmas (Lemmas
\ref{lem:collinearIFFbaseline},\ref{lem:comp2}, \ref{lem:epipoleCa}, \ref{lem:Euclidean camera to quasi}, and \ref{lem:hi in center cones}) which are only present in this section. As a result, some lemmas appear out of order because they are presented in the order they are needed.

\subsection*{\bf Proofs from~\Cref{sec:chiraljointimage}}

Recall from \Cref{def:c_a} that $\mathbf{a}_i = G_i^{-1} \mathbf{p}_i$ 
and $\mathbf{b}_{ij} = G_i^{-1} \mathbf{t}_i - G_j^{-1} \mathbf{t}_j$. Throughout this section, we denote the baseline of finite cameras $A_i$ and $A_j$ by $l_{ij}$, i.e., 
\begin{align}
    l_{ij} := \{\mathbf{q} \in \P^3 : \exists \lambda_1, \lambda_2 \in \R \text{ s.t. } \mathbf{q} = \lambda_1 \mathbf{c}_1 + \lambda_2 \mathbf{c}_2\}
\end{align}

\begin{lemma}
\label{lem:collinearIFFbaseline}
Let $\mathcal{A} = \{A_1,A_2\}$ be a pair of finite cameras $A_i = \begin{bmatrix} G_i & \mathbf{t}_i\end{bmatrix}$ with distinct centers. Fix $\mathbf{q}\in\P^3\setminus\{\mathbf{c}_1,\mathbf{c}_2\}$ and write $(\mathbf{p}_1,\mathbf{p}_2) \sim  \varphi_\mathcal{A}(\mathbf{q})$ with $A_i\mathbf{q} = \lambda_i \mathbf{p}_i$. The vectors $\mathbf{a}_1, \mathbf{a}_2$ and $\mathbf{b}_{12}$ are collinear in $\R^3$ if and only if $\mathbf{q} \in l_{12}$. In this case, $\mathbf{p}_1 = \mathbf{e}_{12}$ and $\mathbf{p}_2 = \mathbf{e}_{21}$.
\end{lemma}

\begin{proof}
We argue geometrically. The vector $(\mathbf{b}_{12}^\top,0)^\top\in \P^3$ is the intersection point of the baseline $l_{12}$ with the hyperplane $L_\infty$. Moreover, a point $\mathbf{p}\in \P^2$ has a $1$-dimensional family of preimages under a finite camera $A$ which is the span of the vector $((G^{-1}\mathbf{p})^\top,0)^\top = (\mathbf{a}^\top,0)^\top \in \P^3$ and the (finite) center of camera $A$, where the camera center cannot be imaged in $A$, of course. Indeed, this follows from the fact that $A(\mathbf{a}^\top,0)^\top = \mathbf{p}$ and the fact that the kernel of $A$ is spanned by its center $\mathbf{c}\in \P^3$. Of course, we could also take the span of any other two points on this line. Below we will see the span of the center $\mathbf{c}$ and a point $\mathbf{q}\in\P^3$ with $A\mathbf{q} \sim \mathbf{p}$.
We now apply these geometric facts to the epipoles and the baseline to show the two implications. 

So first suppose that $\mathbf{a}_1$, $\mathbf{a}_2$, and $\mathbf{b}_{12}$ are collinear. Geometrically, this means that the lines of preimages $\{\mathbf{q}\in\P^3\colon A_i \mathbf{q} \sim \mathbf{p}_i\}$ have the same intersection points with the hyperplane at infinity and that point is also the intersection point of $l_{12}$ and $L_\infty$. Since the baseline contains both centers $\mathbf{c}_1$ and $\mathbf{c}_2$ and the intersection points at infinity coincide, all three lines are equal to the baseline.

Conversely, if $\mathbf{q}$ is on the baseline but not a camera center, then the line spanned by $\mathbf{q}$ and any camera center $\mathbf{c}_i$ is the baseline $l_{ij}$ and the line of preimages $\{\mathbf{q}\in \P^3\colon A_i \mathbf{q} \sim \mathbf{p}_i\}$. \qed
\end{proof}

\begin{lemma}
\label{lem:comp2}
Let $\mathcal{A} = \{A_1,A_2\}$ be a pair of finite cameras with distinct centers. Fix $\mathbf{q}\in\P^3\setminus\{\mathbf{c}_1,\mathbf{c}_2\}$ and write $(\mathbf{p}_1,\mathbf{p}_2) \sim  \varphi_\mathcal{A}(\mathbf{q})$ with $A_i\mathbf{q} = \lambda_i \mathbf{p}_i$. If $\mathbf{q} \notin l_{ij}$, then the following conditions hold:
\begin{align}
    \mathbf{b}_{12}^\top (\mathbf{a}_1 \times \mathbf{a}_2) &= 0, \label{eq:epipolar}\\
    \sign(\lambda_1 q_4) &= \sign\left( (\mathbf{a}_1 \times \mathbf{a}_2 )^\top
    (\mathbf{b}_{12} \times \mathbf{a}_2 ) 
    \right), \label{eq:sign1}\\
    \sign(\lambda_2 q_4) & = \sign \left( 
    (\mathbf{a}_1 \times \mathbf{a}_2 )^\top (\mathbf{b}_{12} \times \mathbf{a}_1 )
    \right), \label{eq:sign2}\\
    \sign(\lambda_1 \lambda_2 ) & =
    \sign\left( ( \mathbf{b}_{12} \times \mathbf{a}_1 )^\top (\mathbf{b}_{12} \times \mathbf{a}_2 )\right) \label{eq:sign3}.
\end{align}
\end{lemma}

\begin{proof} 

Since $\mathbf{q} \neq \mathbf{c}_1, \mathbf{c}_2$, we know $\lambda_1 \neq 0, \lambda_2 \neq 0$. Write $\mathbf{q} = (\mathbf{r}, q_4)$ and $\mathbf{b}$ for $\mathbf{b}_{12}$. 
Eliminating $\mathbf{r}$ by taking the difference of the equations $\la_1\mathbf{p}_1 = G_1\mathbf{r} + q_4 \mathbf{t}_1$ and $\la_2\mathbf{p}_2 = G_2\mathbf{r} + q_4 \mathbf{t}_2$, we get $ \lambda_1 G_1^{-1}\mathbf{p}_1 - \lambda_2 G_2^{-1} \mathbf{p}_2 = q_4 \left(G_1^{-1}\mathbf{t}_1 - G_2^{-1}\mathbf{t}_2\right)$, equivalently,
\begin{align} 
\label{eq:linear}
    \lambda_1 \mathbf{a}_1 - \lambda_2 \mathbf{a}_2 = q_4 \mathbf{b}.
\end{align}
Taking cross products with $\mathbf{a}_1, \mathbf{a}_2$, and $\mathbf{b}$ on both sides of (\ref{eq:linear}), we get
\begin{align} 
-\lambda_2 (\mathbf{a}_2 \times \mathbf{a}_1) = q_4 (\mathbf{b} \times \mathbf{a}_1), \\
\lambda_1  (\mathbf{a}_1 \times \mathbf{a}_2) = q_4  (\mathbf{b} \times \mathbf{a}_2), \\
\lambda_1 (\mathbf{b} \times \mathbf{a}_1)  = \lambda_2 (\mathbf{b} \times \mathbf{a}_2).
\end{align}
We now consider two cases:\\

\noindent \textit{Case a:}
Suppose $q_4 \neq 0$. 
Equation~\eqref{eq:linear} implies that 
$\mathbf{b}$, $\mathbf{a}_1$, and $\mathbf{a}_2$ are coplanar in $\R^3$, 
so that ~\eqref{eq:epipolar} is satisfied. Further, it is straightforward to see from the cross product equations that either $\mathbf{a}_1 \times \mathbf{a}_2$, $\mathbf{b} \times \mathbf{a}_1$ and $\mathbf{b} \times \mathbf{a}_2$ are all equal to zero or none of them are, i.e. either $\mathbf{a}_1, \mathbf{a}_2$, and $\mathbf{b}$ are all pairwise collinear or not. From \Cref{lem:collinearIFFbaseline}, our assumption that $\mathbf{q} \notin l_{ij}$ implies that  $\mathbf{a}_1, \mathbf{a}_2$, and $\mathbf{b}$ are not all collinear. Hence, the equations \eqref{eq:sign1},\eqref{eq:sign2},\eqref{eq:sign3} follow from multiplying each equality above by the transpose of the left hand side. \\

\noindent \textit{Case b:} Suppose $q_4 = 0$. This implies that $\mathbf{a}_1$ and $\mathbf{a}_2$ are collinear in $\R^3$ and $\mathbf{a}_1 \times \mathbf{a}_2 = 0$. This proves~\eqref{eq:epipolar},~\eqref{eq:sign1} and~\eqref{eq:sign2}. Since by 
assumption $\mathbf{b}$ is not collinear with $\mathbf{a}_1$ and $\mathbf{a}_2$~\eqref{eq:sign3} follows by multiplying the third equality above by its right hand side. \qed\end{proof}

\begin{remark}
We comment that \Cref{lem:comp2} above is effectively performing quantifier elimination on the conditions given in \cite[Theorem 4]{WernerPajdla2001}. Indeed there exist scalars $w, \rho > 0$ such that
\begin{align}
    w \mathbf{p}_1 =  G \mathbf{p}_2 + \rho \mathbf{t}
\end{align}
if and only if 
\begin{align}
    \la_1 \mathbf{p}_1 =  \la_2 G \mathbf{p}_2 + q_4 \mathbf{t}
\end{align}
for some scalars $\la_1, \la_2, q_4$ where $\la_1q_4 > 0$, $\la_2 q_4 > 0$,  and $\la_1\la_2 > 0$. We note this equation  is equivalent to \Cref{eq:linear} above for cameras $A_1 = \begin{bmatrix} G & \mathbf{t}\end{bmatrix}$, $A_2 = \begin{bmatrix} I & \mathbf{0}\end{bmatrix}$. We have shown that the signs of these products may be computed directly from the image data $\mathbf{p}_i$ and camera data $A_i$. 
\end{remark}

{
\renewcommand{\thelemma}{\ref{lem:worldpointCa2}}
\addtocounter{lemma}{-1}
\begin{lemma}
Let $\mathcal{A} = \{A_1,\ldots,A_m\}$ be an arrangement of finite cameras such that $D_\mathcal{A}$ is nonempty. If the centers of $\mathcal{A}$ are not collinear, then  
\begin{align}
\mathcal{X}_\mathcal{A} = \mathcal{J}_\mathcal{A}  \cap C_\mathcal{A}. \label{eq:XAiffCA}
\end{align}
If the centers are collinear, then set 
$\mathbf{e} := (\mathbf{e}_1, \dots, \mathbf{e}_m)$ to be the image of the 
common baseline under $\varphi_\mathcal{A}$. 
Then 
\begin{align}
\mathcal{X}_\mathcal{A} \minus \{\mathbf{e}
\} = \left(\mathcal{J}_\mathcal{A}  \cap C_\mathcal{A}\right) \minus  \{\mathbf{e}\}
\end{align}
In both cases, $\overline{\mathcal{X}}_\mathcal{A} \subseteq C_\mathcal{A}$.
\end{lemma}
}

\begin{proof}
Suppose $\mathbf{p} = (\mathbf{p}_1, \dots , \mathbf{p}_m) = \varphi_\mathcal{A}(\mathbf{q})$ for a $\mathbf{q} = (\mathbf{r}, q_4)$ in $\P^3$. As before, we write $\mathbf{b}_{ij} = G_i^{-1}\mathbf{t}_i - G_j^{-1}\mathbf{t}_j$ and $\mathbf{a}_i = G_i^{-1}\mathbf{p}_i$. It suffices to show that $\mathbf{p} \in \mathcal{J}_\mathcal{A} \cap C_\mathcal{A}$ if and only if $\mathbf{p} \in \mathcal{X}_\mathcal{A}$. Recall that the principal ray of camera $A_i$ is given by $\det(G_i) (A_i)_{3\bullet}$ and $p_{i3} = (A_i)_{3,\bullet}\mathbf{q}$. By \Cref{thm:chiral set of an arrangement}, we know $\mathbf{p}\in \mathcal{X}_\mathcal{A}$ if and only if 
\begin{align}
\det(G_i)\lambda_ip_{i3}q_4 &\geq 0 \label{eq:giliq4}\\
\det(G_i)\det(G_j)\lambda_ip_{i3}\lambda_jp_{j3} &\geq 0 \label{eq:gigjlilj}
\end{align}
for all $i,j$. We will show that $\mathbf{p}$ satisfies these inequalities if and only if $\mathbf{p} \in \mathcal{J}_\mathcal{A} \cap C_\mathcal{A}$, i.e., if $\mathbf{p}$ satisfies the inequalities
\begin{align}
\det(G_i)p_{i3}(\mathbf{a}_i \times \mathbf{a}_j )^\top
    (\mathbf{b}_{ij} \times \mathbf{a}_i ) &\geq 0 \label{eq:cagi}  \\
\det(G_i)\det(G_j)p_{i3}p_{j3} ( \mathbf{b}_{ij} \times \mathbf{a}_i )^\top (\mathbf{b}_{ij} \times \mathbf{a}_j) &\geq 0 \label{eq:cagigj} 
\end{align}
for all $i,j$.
We make some observations that follow for all $\mathbf{q}$ when the cameras in $\mathcal{A}$ are noncollinear. \begin{enumerate}
    \item For each camera $A_i$, there is some camera $A_j$ such that $\mathbf{q} \notin l_{ij}$. Indeed, fixing $i$, the pencil of lines $l_{ij}$ are not all identical, hence $\mathbf{q}$ cannot lie on all of them.  \Cref{lem:comp2} therefore implies that 
    \begin{align}
    \sign(\la_iq_4) = \sign( \mathbf{a}_i \times \mathbf{a}_j)^\top(\mathbf{b}_{ij} \times \mathbf{a}_i) .
\end{align}
    \item For each pair of cameras $\{A_i,A_j\}$ such that $\mathbf{q} \in l_{ij}$, there is a camera $A_k$ such that $\mathbf{q} \notin l_{ik}$ and $\mathbf{q} \notin l_{jk}$. This follows from noncollinearity because for every line $l_{ij}$, there must be some center $\mathbf{c}_k$ not on this line. \Cref{lem:comp2} therefore implies that 
\begin{align}
    \sign(\la_i\la_j) &= \sign(\la_i\la_j\la_k^2) \\
                     &= \sign(\la_i\la_k)\sign(\la_j\la_k) \\
                     &= \sign( \mathbf{b}_{ik} \times \mathbf{a}_i )^\top (\mathbf{b}_{ik} \times \mathbf{a}_k)* \nonumber \\
                     &( \mathbf{b}_{jk} \times \mathbf{a}_j )^\top (\mathbf{b}_{jk} \times \mathbf{a}_k).
\end{align}
\end{enumerate}
Suppose $\mathbf{p}$ satisfies the inequalities (\ref{eq:giliq4}) and (\ref{eq:gigjlilj}). Then for every pair $i,j$ either 
\begin{align}
\det(G_i)p_{i3}(\mathbf{a}_i \times \mathbf{a}_j )^\top(\mathbf{b}_{ij} \times \mathbf{a}_i ) = 0 \text{ if $\mathbf{q} \in l_{ij}$ or} \\
\sign \left(\det(G_i)p_{i3}(\mathbf{a}_i \times \mathbf{a}_j )^\top
    (\mathbf{b}_{ij} \times \mathbf{a}_i )\right)=&  \nonumber \\
    \sign(\det(G_i)\lambda_ip_{i3}q_4 ) \ge &  0
\end{align}
Similar reasoning shows that (\ref{eq:cagigj}) holds for all $i,j$. 
Conversely, suppose $\mathbf{p}$ satisfies the inequalities (\ref{eq:cagi}) and (\ref{eq:cagigj}). From the observations above, we see that $\sign(\la_iq_4)$ and $\sign(\la_i\la_j)$ can be inferred from the inequalities (\ref{eq:cagi}) and (\ref{eq:cagigj}). Hence (\ref{eq:giliq4}) and (\ref{eq:gigjlilj}) hold, meaning $\mathbf{p} \in\mathcal{X}_\mathcal{A}$. We conclude that  
\[
\mathcal{J}_\mathcal{A} \cap C_\mathcal{A} =  \mathcal{X}_\mathcal{A}
\]
In particular, this means $\mathcal{X}_\mathcal{A} \subseteq C_\mathcal{A}$, so $\overline{\mathcal{X}_\mathcal{A}} \subseteq C_\mathcal{A}$ because $C_\mathcal{A}$ is closed. 
The above argument holds for all $\mathbf{p}$ such that its preimage under $\varphi_\mathcal{A}$ is a unique $\mathbf{q}$. For collinear cameras this is true for all $\mathbf{p} \neq \mathbf{e}$, hence the only point that must be removed from (\ref{eq:XAiffCA}) is $\mathbf{e}$. \qed
\end{proof}

\begin{lemma}
\label{lem:epipoleCa}
Let $\mathcal{A} = \{A_1,A_2\}$ be a pair of finite cameras.
If $\mathbf{c}_j$ has nonnegative depth in the other camera $A_i$, then $E_j$ is contained in $C_\mathcal{A}$. Otherwise $(\mathbf{e}_{12},\mathbf{e}_{21})$ is the only point in $E_j$ that lies in $C_\mathcal{A}$. 
\end{lemma}

\begin{proof}
Without loss of generality, we can assume $j=2$.
Let  $A_i = \begin{bmatrix} G_i & \mathbf{t}_i \end{bmatrix}$ for $i=1,2$. We write $\tilde{\mathbf{c}}_i = -G_i^{-1} \mathbf{t}_i$ and $\mathbf{b} = \tilde{\mathbf{c}}_2 - \tilde{\mathbf{c}}_1$. 
Let $\mathbf{s}_1 = G_1\tilde{\mathbf{c}}_2 + \mathbf{t}_1 = G_1\left( -G_2^{-1}\mathbf{t}_2 + G_1^{-1} \mathbf{t}_1\right)$ $ = G_1\mathbf{b}$. Then, the image of $\mathbf{c}_2$ in $A_1$ is $\mathbf{e}_{12} = \lambda_1 \mathbf{s}_1$. Similarly, let $\mathbf{s}_2 = G_2 \tilde{\mathbf{c}}_1 + \mathbf{t}_2 = G_2(-G_1^{-1} \mathbf{t}_1 + G_2^{-1}\mathbf{t}_2) = -G_2\mathbf{b}$. Then, the image of $\mathbf{c}_1$ in $A_2$ is $\mathbf{e}_{21} = \lambda_2 \mathbf{s}_2$.

Now if $\mathbf{p}_1 = \mathbf{e}_{12} = \lambda_1  G_1\mathbf{b}$, then $\mathbf{a}_1 = \lambda_1 \mathbf{b}$ and $\mathbf{b} \times \mathbf{a}_1 = 0$. Which means that the only inequality defining $C_\mathcal{A}$ not identically equal to zero is 
\begin{align}\label{eq:theoneofthree}
   &  \det(G_1)p_{13} (\mathbf{a}_1\times \mathbf{a}_2)^\top (\mathbf{b} \times \mathbf{a}_2) \geq 0.
\end{align}

Plugging $\mathbf{a}_1 = \lambda_1 \mathbf{b}$ and $\mathbf{p}_1 = \mathbf{e}_{12} = \lambda_1 \mathbf{s}_1$ in the above we get
\begin{align}
\det(G_1)\lambda_1 s_{13} (\lambda_1 \mathbf{b} \times \mathbf{a}_2)^\top (\mathbf{b} \times \mathbf{a}_2) \geq 0\\
\det(G_1) \lambda_1^2 s_{13}\|\mathbf{b} \times \mathbf{a}_2\|^2 \geq 0\\
\det(G_1) s_{13}\|\mathbf{b} \times \mathbf{a}_2\|^2 \geq 0
\end{align}
This can be satisfied in two ways, namely $\mathbf{b} \times \mathbf{a}_2 = 0$ or $\det(G_1) s_{13} \geq 0$.\\

\noindent \textit{Case 1:} Suppose $\mathbf{b} \times \mathbf{a}_2 = 0$. Then since 
    \begin{align}
      \mathbf{b} \times \mathbf{a}_2 = 0
\iff  \mathbf{a}_2 \sim \mathbf{b} 
\iff  \mathbf{p}_2 \sim G_2 \mathbf{b} \sim \mathbf{e}_{21},
\end{align}
the condition $\mathbf{b}\times \mathbf{a}_2 = 0$ is the same as 
$\mathbf{p}_2 \sim \mathbf{e}_{21}$.\\

\noindent \textit{Case 2:} Now suppose $\det(G_1) s_{13} \geq 0$. Observe that the depth of $\mathbf{c}_2$ in $A_1$ is
\begin{align}
    \operatorname{depth} (\mathbf{c}_2; A_1) = \frac{1}{|\det(G_1)|
    \|G^1_{3,\bullet}\|
    } \det(G_1)s_{13}.
\end{align}

Therefore, $\det(G_1) s_{13}\geq 0$ if and only if $\mathbf{c}_2$ has nonnegative depth in $A_1$. In this case, the inequality \eqref{eq:theoneofthree} imposes no constraints on $\mathbf{p}_2$ as claimed.
\qed\end{proof}

{
\renewcommand{\thelemma}{\ref{lem:eaplus}}
\addtocounter{lemma}{-1}
\begin{lemma}
Let $\mathcal{A} = \{A_1,\ldots,A_m\}$ be an arrangement of finite cameras. If the centers $\mathbf{c}_i$ are not collinear then $E_\mathcal{A} \cap C_\mathcal{A} = E^+_\mathcal{A}$.  Otherwise, 
 $E_\mathcal{A} \cap C_\mathcal{A} = E^+_\mathcal{A} \cup \{\mathbf{e}
 \}$.
\end{lemma}
}

\begin{proof}
We first show that $E_\mathcal{A}^+$ is in $C_\mathcal{A}$.
Since the inequalities defining $C_\mathcal{A}$ only depend on pairs of cameras, we can restrict to the case of every pair $\{A_k,A_\ell\}$. If none of the indices are equal to $j$ the cameras $A_k$ and $A_\ell$ see the center of camera $A_j$ and the $C_{\{A_k,A_\ell\}}$ inequalities are satisfied if $\mathbf{c}_j\in D_\mathcal{A}$. If one of the indices is equal to $j$, we use the previous \Cref{lem:epipoleCa}. We conclude that if $\mathbf{c}_j\in D_\mathcal{A}$, then \[
E_j\subseteq \bigcap_{k,\ell} C_{\{A_k, A_\ell\}} = C_\mathcal{A}.
\]

Conversely, we argue that $\mathbf{p} \in E_\mathcal{A} \setminus E^+_\mathcal{A}$ cannot lie in $C_\mathcal{A}$. This means $\mathbf{p} = (\mathbf{e}_{1j}, \dots, \mathbf{p}_j, \dots, \mathbf{e}_{mj})$ where $\mathbf{c}_j$ has negative depth in some camera $A_k \in \mathcal{A}$. From the definition of depth, this means \[
\mathbf{n}_k^\top \mathbf{c}_j = \det(G_k)\la_3p_{k3}c_{j4} < 0.
\]
If the camera centers are not collinear, we can choose a camera $A_\ell$ with $\ell\neq j$ such that $\mathbf{c}_j$, $\mathbf{c}_k$, and $\mathbf{c}_\ell$ do not lie on a line. By \Cref{lem:comp2}, 
\begin{align}
    &\sign \left(\det(G_k)p_{k3} (\mathbf{a}_k\times\mathbf{a}_\ell)^\top(\mathbf{b}_{k\ell}\times \mathbf{a}_\ell) \right)\\
    = &\sign(\det(G_k)p_{k3}\la_kc_{j4}) < 0, 
\end{align} 
violating one of the inequalities of $C_\mathcal{A}$. Hence, $\mathbf{p}\notin C_\mathcal{A}$.


If the camera centers are collinear, then the point of epipoles
$(\mathbf{e}_1,\mathbf{e}_2,\ldots,\mathbf{e}_m)$ is the image of the line connecting the centers. This point trivially lies in $C_\mathcal{A}$ because all 
defining inequalities evaluate to $0$ on this point. Again,  let $A_k$ be a camera such that $\mathbf{c}_j$ has negative depth in $A_k$. \Cref{lem:epipoleCa} shows that 
$(\mathbf{e}_{i},\mathbf{e}_k)$ is the only point in $E_i$
that lies in $C_{ \{A_i,A_k\}}$.
\qed\end{proof}

{
\renewcommand{\thelemma}{\ref{lem:eaplusplus}}
\addtocounter{lemma}{-1}
\begin{lemma}
Let $\mathcal{A} = \{A_1,A_2,\ldots,A_m\}$ be an arrangement of finite cameras with distinct centers. Let $E_\mathcal{A}^{++}$ be the union of the sets $E_j$ such that $\mathbf{c}_j$ has positive depth in every camera $A_i\in \mathcal{A}\setminus\{A_j\}$, then 
$E_\mathcal{A}^{++}\subseteq \overline{\mathcal{X}}_\mathcal{A}$.
\end{lemma}
}
\begin{proof}
We can approach $\mathbf{p}=(\mathbf{p}_1,\mathbf{p}_2,\ldots,\mathbf{p}_m)\in E_j$ by the sequence of points $\varphi_\mathcal{A}(\mathbf{v}(s))$ as $s$ goes to $0$, where
$$\mathbf{v}(s) = 
\begin{pmatrix} 
sG_j^{-1}\mathbf{p}_j \\ 0 \end{pmatrix} + \mathbf{c}_j.$$ 
Indeed, $A_i \mathbf{v}(s) = s G_i G_j^{-1} \mathbf{p}_j +  \mathbf{e}_{ij}$ which approaches $\mathbf{e}_{ij}$ as $s \rightarrow 0$, and $A_j \mathbf{v}(s) =  s \mathbf{p}_j \sim \mathbf{p}_j$ for all $s \neq 0$.
Since depth is continuous and $\mathbf{c}_j$ has positive depth in $A_i$, $i \neq j$, 
the point $\mathbf{v}(s)$ has positive depth in $A_i$ for sufficiently small $s\in \R$. The depth of $\mathbf{v}(s)$ with respect to camera $A_j$ changes sign at $s=0$ (or it is identically $0$, that is $\mathbf{v}(s)$ lies on the principal plane of camera $A_j$ for all $s$). 
Therefore, $\mathbf{v}(s)$ is in $D_\mathcal{A}$ for sufficiently small positive or negative $s$.
So $\mathbf{p}$ lies in the closure of $\mathcal{X}_\mathcal{A}$.
\qed\end{proof}

\subsection*{\bf Proofs from~\Cref{sec:projectiveexistence}}
For a camera $A = \begin{bmatrix} G \,& \,{\mathbf t} \end{bmatrix}$, one can compute a
kernel element $\mathbf{c}$ via Cramer's rule so that $c_i$ is $(-1)^{i}$ times 
the determinant of the submatrix of $A$ obtained by dropping the $i$th column. In particular, $c_4 = \det(G)$. We call 
this representation of the center, the {\em Cramer's rule center} of $A$, and denote it as 
$\mathbf{c}$. Recall the representative
$$\mathbf{c}_A = \begin{bmatrix} -G^{-1} \mathbf{t} \\ 1 \end{bmatrix}$$
obtained by scaling of the Cramer's rule center by $\det(G)$.

{
\renewcommand{\thelemma}{\ref{lem:transformations}}
\addtocounter{lemma}{-1}
\begin{lemma}
Let $A = \begin{bmatrix} G \,& \,{\mathbf t} \end{bmatrix}$ be a finite camera with center ${\mathbf{c}}_A$. 
Let $H \in \GL_4$ with fourth row $\mathbf{h}^\top$ and $\delta = \det(H^{-1})$. Then 
\begin{enumerate}
    \item Under the homography $\mathbf{q} \mapsto H \mathbf{q}$, the plane $\mathbf{h}^\top \bf{q} = 0$ maps to the plane at infinity.
    \item The camera $AH^{-1}$ is finite if and only if $\mathbf{h}^\top \mathbf{c}_A\neq 0$. Its center then is ${\mathbf c}_{AH^{-1}} = \frac{1}{{\mathbf h}^\top {{\mathbf c}}_A} H {{\mathbf c}}_A$.
    \item The principal ray of $AH^{-1}$ is $$
    \mathbf{n}_{AH^{-1}} = \delta ( \mathbf{h}^\top {\mathbf{c}}_A )  H^{-\top}\mathbf{n}_A, $$
    and so for all $\mathbf{q} \in \R^4$, we have 
    $$
    \mathbf{n}_{AH^{-1}}^\top(H\mathbf{q}) = \delta (\mathbf{h}^\top \mathbf{c}_A ) (\mathbf{n}_{A}^\top \mathbf{q}).
    $$
\end{enumerate}
\end{lemma}
}

\begin{proof} 
\begin{enumerate}
    \item A point $H\mathbf{q}$ lies on the plane at infinity if and only if $(H\mathbf{q})_4 = \mathbf{h}^\top \mathbf{q} = 0$.
    \item The first equivalence in the claim follows from the previous part. Let $\mathbf{c}$ be the Cramer's rule center of $A$. Then, ${\mathbf c}_A = \frac{1}{\det(G)} {\mathbf c}$. Observe that $H {\mathbf c}$ is a representative for the center of $AH^{-1}$ with $(H\mathbf{c})_4 = \bf{h}^\top \bf{c}$. We compute
    \[
    {\mathbf c}_{AH^{-1}} = \frac{1}{{\mathbf h}^\top {{\mathbf c}}} H{{\mathbf c}} = \frac{\det(G)}{\det(G){\mathbf h}^\top {{\mathbf c}_A}} H{{\mathbf c}_A} ,
    \]  
    from which the result follows.
    \item  The determinant of the first $3 \times 3$ block of $AH^{-1}$ is the last coordinate of the Cramer's rule center of $AH^{-1}$. Hartley \cite{hartley1998chirality} shows that the Cramer's rule center of $AH^{-1}$ is $\delta H {\mathbf c} = \delta \det(G) H {\mathbf c}_A$. The principal ray of $AH^{-1}$ is therefore,
    \begin{align}
        \mathbf{n}_{AH^{-1}} &= \left(\delta H \mathbf{c}\right)_4  (A_{3\bullet} H^{-1})^\top \\
                        &= \delta \det(G) ({\mathbf h}^\top {\mathbf c}_A) H^{-\top} A_{3\bullet}^\top \\
                        &= \delta (\mathbf{h}^\top {\mathbf{c}}_A ) H^{-\top} \mathbf{n}_A
    \end{align}
 Plugging in the expression for the principal ray, we compute
\begin{align}
\label{eq:Haffectingproduct}
 (\mathbf{n}_{AH^{-1}}^\top(H\mathbf{q})) &=  \delta (\mathbf{h}^\top {\mathbf{c}}_A ) (\mathbf{n}_{A}^\top H^{-1} H \mathbf{q}) \\
 &= \delta (\mathbf{h}^\top {\mathbf{c}}_A ) (\mathbf{n}_{A}^\top \mathbf{q}) 
\end{align}
for all $\mathbf{q}\in \R^4$.
\end{enumerate}
\end{proof}

\subsection*{\bf Proofs from~\Cref{sec:euclideanexistence}}

Applying techniques from \Cref{sec:projectiveexistence}, we show when a Euclidean reconstruction can be made chiral using a homography. As we argue in \Cref{sec:euclideanexistence}, we may assume that our starting and target reconstructions have $A_1 = \begin{bmatrix} I & \mathbf{0} \end{bmatrix}$. This choice of the first cameras restricts the homographies we need to consider to 
$H$ such that $H^{-1} = \begin{bmatrix} I & 0 \\ \mathbf{v}^\top & \de \end{bmatrix}$ for some $\mathbf{v} \in \R^3$ and nonzero $\de \in \R$. Note that $\delta = \det{H}^{-1}$.

We now introduce the notion of a quasi-Euclidean camera.
\begin{definition}
A camera $A = \begin{bmatrix} U & \mathbf{t} \end{bmatrix} $ is \emph{quasi-Euclidean} if  $UU^\top = I$.
\end{definition}

While we are interested in transforming a Euclidean reconstruction into a chiral Euclidean reconstruction, a homography may only be able to yield a reconstruction where the transformed cameras are quasi-Euclidean. However, since scaling a camera does not change 
chirality, a chiral quasi-Euclidean reconstruction can be turned into a chiral Euclidean reconstruction by multiplying $A_i$ by $\sign(\det(U_i))$. 
As a result, we only need to search for a homography $H$ that sends our starting Euclidean reconstruction to one where every camera is quasi-Euclidean, which bring us to the following lemma.

\begin{lemma}
\label{lem:Euclidean camera to quasi}
Given a Euclidean camera $A = \begin{bmatrix} R &\mathbf{t}\end{bmatrix} $ such that $\mathbf{t} \neq \mathbf{0}$ and a homography $H$ such that $H^{-1} = \begin{bmatrix} I & 0 \\ \mathbf{v}^\top & \de \end{bmatrix}$ for some vector $\mathbf{v} \in \R^3$ and $\de \neq 0$, the camera $AH^{-1}$ is quasi-Euclidean if and only if $\mathbf{v} = \mathbf{0}$ or $\mathbf{v} = -\frac{2}{\|\mathbf{t}\|^2} R^\top \mathbf{t}$. 
\end{lemma}
\begin{proof}
The requirement that  $AH^{-1}$ be
quasi-Euclidean translates to 
 \begin{align*}
    I &= (R + \mathbf{t}\mathbf{v}^\top)^\top(R + \mathbf{t}\mathbf{v}^\top) \\
    &= R^\top R + \mathbf{v}\mathbf{t}^\top R + R^\top\mathbf{t}\mathbf{v}^\top + \mathbf{v}\mathbf{t}^\top \mathbf{t} \mathbf{v}^\top \\
        &= I + \mathbf{v}\mathbf{t}^\top R + R^\top\mathbf{t}\mathbf{v}^\top +  \|\mathbf{t}\|^2 \mathbf{v}\mathbf{v}^\top
 \end{align*}
 For the fixed vector $\tilde{\mathbf{c}} : = -R^\top\mathbf{t} \neq \mathbf{0}$, this system is equivalent to finding $\mathbf{v}$ such that $M := -\mathbf{v}\tilde{\mathbf{c}}^\top - \tilde{\mathbf{c}}\mathbf{v}^\top + (\mathbf{v}\tilde{\mathbf{c}}^\top )(\tilde{\mathbf{c}}\mathbf{v}^\top) = 0$. Certainly $\mathbf{v} = \mathbf{0}$ is one solution. Otherwise, applying $M$ to $\mathbf{v}$, we get that 
\begin{align}
{\mathbf 0} = M\mathbf{v} &= -(\tilde{\mathbf{c}}^\top\mathbf{v}) \mathbf{v}  -(\mathbf{v}^\top\mathbf{v}) \tilde{\mathbf{c}} + (\tilde{\mathbf{c}}^\top\tilde{\mathbf{c}})(\mathbf{v}^\top\mathbf{v}) \mathbf{v} \\
        &= ((\tilde{\mathbf{c}}^\top\tilde{\mathbf{c}})(\mathbf{v}^\top\mathbf{v}) -(\tilde{\mathbf{c}}^\top\mathbf{v})) \mathbf{v} - (\mathbf{v}^\top\mathbf{v}) \tilde{\mathbf{c}}.\label{eq:v is multiple of c}
\end{align}
If $(\tilde{\mathbf{c}}^\top\tilde{\mathbf{c}})(\mathbf{v}^\top\mathbf{v}) -(\tilde{\mathbf{c}}^\top\mathbf{v}) = \mathbf{0}$ for some $\mathbf{v} \neq 0$, then $M\mathbf{v} = (\mathbf{v}^\top \mathbf{v}) \tilde{\mathbf{c}} \neq \mathbf{0}$. Therefore, \Cref{eq:v is multiple of c} implies that 
$\mathbf{v} = \lambda \tilde{\mathbf{c}}$ for some $\lambda \neq 0$. Solving for $\lambda$, we get $\lambda = \frac{2}{\tilde{\mathbf{c}}^\top\tilde{\mathbf{c}}}$
Which gives us the only additional solution $\mathbf{v} = \frac{2}{\|\tilde{\mathbf{c}}\|^2}  \tilde{\mathbf{c}} = -\frac{2}{\|\mathbf{t}\|^2} R^\top \mathbf{t}$. 
\qed\end{proof}

Without loss of generality, we may assume the homographies in \Cref{lem:Euclidean camera to quasi} have $|\de| = 1$, leaving us with the following four possibilities for two view Euclidean reconstructions:

\begin{align}
H_1^{-1}  := \begin{bmatrix} I & \mathbf{0} \\ \mathbf{0}^\top & 1 \end{bmatrix}, \
H_2^{-1}  := \begin{bmatrix} I & \mathbf{0} \\ \mathbf{0}^\top & -1 \end{bmatrix}, \\
H_3^{-1} := \begin{bmatrix} I & \mathbf{0} \\ \mathbf{v}^\top & 1 \end{bmatrix}, \
H_4^{-1} := \begin{bmatrix} I & \mathbf{0} \\ \mathbf{v}^\top & - 1 \end{bmatrix} 
\end{align}
where $\mathbf{v} = -\frac{2}{\|\mathbf{t}\|^2} R^\top \mathbf{t}$. These have the following inverses.
\begin{align}
H_1  = \begin{bmatrix} I & \mathbf{0} \\ \mathbf{0}^\top & 1 \end{bmatrix}, \quad H_2  = \begin{bmatrix} I & \mathbf{0} \\ \mathbf{0}^\top & -1 \end{bmatrix}, \\  
H_3 = \begin{bmatrix} I & \mathbf{0} \\ -\mathbf{v}^\top & 1 \end{bmatrix}, \quad H_4 = \begin{bmatrix} I & \mathbf{0} \\ \mathbf{v}^\top & - 1 \end{bmatrix} 
\end{align}
A Euclidean reconstruction $(\{A_1 = \begin{bmatrix} I & \mathbf{0}\end{bmatrix} ,A_2 = \begin{bmatrix} R &\mathbf{t}\end{bmatrix} \}, \mathcal{Q})$ can be made chiral if and only if one of $(\mathcal{A}H_i^{-1}, H_i \mathcal{Q})$ is chiral. Just as in the projective case, we assume we start with a signed reconstruction. Let $\mathbf{h}_i$ be the last row of $H_i$. From \Cref{thm:cone conditions for chiral recon}, we know we need only check if one of $\mathbf{h}_i$ lies in the cone intersection $K_\mathcal{Q}^\ast \cap (\interior K_{\si C}^\ast \cup \interior K_{-\si C}^\ast)$. As the following lemma shows, the special structure of $\mathbf{h}_i$ causes the cone conditions to simplify.

\begin{lemma}
\label{lem:hi in center cones}
Let $(\{A_1 = \begin{bmatrix} I & \mathbf{0}\end{bmatrix} ,A_2 = \begin{bmatrix} R &\mathbf{t}\end{bmatrix} \}, \mathcal{Q})$ be a signed Euclidean reconstruction of $\mathcal{P}$ such that $\mathbf{t} \neq \mathbf{0}$. 
\begin{enumerate}
    \item If $\si_1 = \si_2$, then $\mathbf{h}_1, \mathbf{h}_2 \in \interior K_{\si C}^\ast \cup \interior K_{-\si C}^\ast$ and $\mathbf{h}_3, \mathbf{h}_4 \notin \interior K_{\si C}^\ast \cup \interior K_{-\si C}^\ast$.
    \item If $\si_1 \neq \si_2$ then $\mathbf{h}_3, \mathbf{h}_4 \in \interior K_{\si C}^\ast \cup \interior K_{-\si C}^\ast$ and $\mathbf{h}_1, \mathbf{h}_2 \notin \interior K_{\si C}^\ast \cup \interior K_{-\si C}^\ast$.
\end{enumerate}
\end{lemma}

\begin{proof}
We first compute $\mathbf{h}_i^\top \si_j \mathbf{c}_j$ for all $i,j$:
\begin{align}
    \mathbf{h}_1^\top \si_1 \mathbf{c}_1 &= \si_1,\; \mathbf{h}_1^\top \si_2 \mathbf{c}_2 = \si_2\\
    \mathbf{h}_2^\top \si_1 \mathbf{c}_1 &= -\si_1, \; \mathbf{h}_2^\top \si_2 \mathbf{c}_2 = -\si_2\\
    \mathbf{h}_3^\top \si_1 \mathbf{c}_1 &= \si_1, \; \mathbf{h}_3^\top \si_2 \mathbf{c}_2 = (-\mathbf{v}^\top (-R^\top \mathbf{t}) + 1) \si_2 = -\si_2  \\
    \mathbf{h}_4^\top \si_1 \mathbf{c}_1 &= -\si_1, \; 
    \mathbf{h}_4^\top \si_2 \mathbf{c}_2 = (\mathbf{v}^\top (-R^\top \mathbf{t}) - 1) \si_2 = \si_2 
\end{align}
The vectors $\mathbf{h}_1$ and $\mathbf{h}_2$ make the same sign inner product with $\si_1 \mathbf{c}_1$ and $\si_2 \mathbf{c}_2$ if and only if $\si_1 = \si_2$. Similarly the vectors $\mathbf{h}_3$ and $\mathbf{h}_4$ make the same sign inner product with $\si_1 \mathbf{c}_1$ and $\si_2 \mathbf{c}_2$ if and only if $\si_1 = - \si_2$.
\qed\end{proof}

{
\renewcommand{\thetheorem}{\ref{thm:calibrated-2-views}}
\addtocounter{theorem}{-1}
\begin{theorem}
Let $(\{A_1 = \begin{bmatrix} I & \mathbf{0}\end{bmatrix} ,A_2 = \begin{bmatrix} R &\mathbf{t}\end{bmatrix} \}, \mathcal{Q})$ be a signed Euclidean reconstruction of $\mathcal{P}$ with distinct centers. There exists a 
chiral Euclidean reconstruction of $\mathcal{P}$ if and only if $\mathbf{n}_\infty \in K_\mathcal{Q}^\ast \cup K_{-\mathcal{Q}}^\ast$ or $\mathbf{r} := \begin{bmatrix} -\frac{2}{\|\mathbf{t}\|^2} R^\top \mathbf{t} \\1 \end{bmatrix} \in K_\mathcal{Q}^\ast \cup K_{-\mathcal{Q}}^\ast$. Equivalently, if exactly one of the following holds for all $\mathbf{q}_i$:
\[
q_{i4} \ge 0 \;\text{ or } \; q_{i4} \le 0 \; \text{ or } \; \mathbf{r}^\top \mathbf{q}_i \ge 0  \; \text{ or } \; \mathbf{r}^\top \mathbf{q}_i \le 0.
\]
\end{theorem}
}

\begin{proof}
By \Cref{thm:cone conditions for chiral recon}, a chiral Euclidean reconstruction exists if and only if one of the $\mathbf{h}_i$ lies in the cone intersection $K_\mathcal{Q}^\ast \cap (\interior K_{\si C}^\ast \cup \interior K_{-\si C}^\ast)$. By \Cref{lem:hi in center cones}, if $\si_1 = \si_2$, it is necessary and sufficient that either $\mathbf{h}_1 = \mathbf{n}_\infty \in K_\mathcal{Q}^\ast$ or $\mathbf{h}_2 = -\mathbf{n}_\infty \in K_\mathcal{Q}^\ast$. On the other hand, if $\si_1 \neq \si_2$, it is necessary and sufficient that either $\mathbf{h}_3 = \begin{bmatrix} -\frac{2}{\|\mathbf{t}\|^2} R^\top \mathbf{t} \\1 \end{bmatrix}  \in K_\mathcal{Q}^\ast$ or $\mathbf{h}_4 = -\begin{bmatrix} -\frac{2}{\|\mathbf{t}\|^2} R^\top \mathbf{t} \\1 \end{bmatrix}  \in K_\mathcal{Q}^\ast$, proving the statement.
\qed\end{proof}

{
\renewcommand{\thetheorem}{\ref{thm:calibrated-m-views}}
\addtocounter{theorem}{-1}
\begin{theorem}
Let $(\mathcal{A}, \mathcal{Q})$ be a signed Euclidean reconstruction of $\mathcal{P}$ with $m>2$ cameras, distinct centers, and $A_1 = \begin{bmatrix} I& \mathbf{0}\end{bmatrix}$. There exists a chiral Euclidean reconstruction of $\mathcal{P}$ if and only if $\si_i = \si_j$ for all $1\le i<j \le m$ and either $q_{i4} \ge 0$ for all $i$ or $q_{i4} \le 0$ for all $i$.
\end{theorem}
}

\begin{proof}
Since the cameras have distinct centers, the vectors $-\frac{2}{\|\mathbf{t}_i\|^2} R_i^\top \mathbf{t}_i $ will not coincide, so by \Cref{lem:Euclidean camera to quasi}, the only homographies we can consider are $H_1$ and $H_2$. As in \Cref{lem:hi in center cones}, $\mathbf{h}_1 = \mathbf{n}_\infty, \mathbf{h}_2 = -\mathbf{n}_\infty \in \interior K_{\si C}^\ast \cup \interior K_{-\si C}^\ast$ if and only if $\si_i = \si_j$ for all $i,j$. When this is the case, a chiral reconstruction exists if and only if $\mathbf{n}_\infty \in K_\mathcal{Q}^\ast$ or $-\mathbf{n}_\infty \in K_\mathcal{Q}^\ast$, proving the statement. 
\qed\end{proof}

\clearpage
\bibliographystyle{plain}
\bibliography{references}

\begin{thebibliography}{10}

\bibitem{idealsofthemultiviewvariety}
Sameer Agarwal, Andrew Pryhuber, and Rekha~R Thomas.
\newblock Ideals of the multiview variety.
\newblock {\em IEEE Transactions on Pattern Analysis and Machine Intelligence},
  43(4):1279--1292, 2021.

\bibitem{AST11}
Chris Aholt, Bernd Sturmfels, and Rekha Thomas.
\newblock A {H}ilbert scheme in computer vision.
\newblock {\em Canadian Journal of Mathematics}, 65(5):961--988, 2013.

\bibitem{boyd2004convex}
Stephen Boyd and Lieven Vandenberghe.
\newblock {\em Convex Optimization}.
\newblock Cambridge University Press, 2004.

\bibitem{FLP01}
Olivier Faugeras, Quang-Tuan Luong, and T.~Papadopoulou.
\newblock {\em The Geometry of Multiple Images: The Laws that Govern the
  Formation of Images of a Scene and Some of Their Applications}.
\newblock MIT Press, 2001.

\bibitem{HartleyZisserman2004}
R.~I. Hartley and A.~Zisserman.
\newblock {\em Multiple View Geometry in Computer Vision}.
\newblock Cambridge University Press, second edition, 2004.

\bibitem{hartley1998chirality}
Richard~I. Hartley.
\newblock Chirality.
\newblock {\em International Journal of Computer Vision}, 26(1):41--61, 1998.

\bibitem{HA97}
A.~{Heyden} and K.~{{A}str{\"o}m}.
\newblock Algebraic properties of multilinear constraints.
\newblock {\em Mathematical Methods in the Applied Sciences}, 20:1135--1162,
  September 1997.

\bibitem{laveau1996oriented}
St{\'e}phane Laveau and Olivier Faugeras.
\newblock Oriented projective geometry for computer vision.
\newblock In {\em European Conference on Computer Vision}. Springer, 1996.

\bibitem{existenceprojectivereconstruction}
Hon{-}Leung Lee.
\newblock On the existence of a projective reconstruction.
\newblock {\em CoRR}, abs/1608.05518, 2016.

\bibitem{longuet1981computer}
H~Christopher Longuet-Higgins.
\newblock A computer algorithm for reconstructing a scene from two projections.
\newblock {\em Nature}, 293(5828):133, 1981.

\bibitem{YM12}
Yi~Ma, Stefano Soatto, Jana Kosecka, and S~Shankar Sastry.
\newblock {\em An Invitation to 3-d Vision: From Images to Geometric Models}.
\newblock Springer, 2012.

\bibitem{M12}
Stephen Maybank.
\newblock {\em Theory of Reconstruction from Image Motion}.
\newblock Springer-Verlag, 1993.

\bibitem{mumford1996red}
David Mumford.
\newblock {\em The Red Book of Varieties and Schemes}, volume 1358.
\newblock 1996.

\bibitem{nister2004untwisting}
David Nist{\'e}r.
\newblock Untwisting a projective reconstruction.
\newblock {\em International Journal of Computer Vision}, 60(2):165--183, 2004.

\bibitem{nisterschaffalitzky}
David Nist{\'e}r and Frederik Schaffalitzky.
\newblock Four points in two or three calibrated views: Theory and practice.
\newblock {\em International Journal of Computer Vision}, 67(2):211--231, 2006.

\bibitem{stolfi1991oriented}
J.~Stolfi.
\newblock {\em Oriented Projective Geometry: A Framework for Geometric
  Computations}.
\newblock Academic Press, 1991.

\bibitem{THP15}
Matthew Trager, Martial Hebert, and Jean Ponce.
\newblock The joint image handbook.
\newblock In {\em IEEE International Conference on Computer Vision}, 2015.

\bibitem{triggs1995geometry}
B~Triggs.
\newblock The geometry of projective reconstruction {I}: Matching constraints
  and the joint image.
\newblock {\em Unpublished}, 1995.

\bibitem{triggs95}
B.~Triggs.
\newblock Matching constraints and the joint image.
\newblock In {\em IEEE International Conference on Computer Vision}, pages
  338--343, 1995.

\bibitem{werner2003combinatorial}
Tomas Werner.
\newblock Combinatorial constraints on multiple projections of a set of points.
\newblock In {\em IEEE International Conference on Computer Vision}, pages
  1011--1016, 2003.

\bibitem{werner2003constraint}
Tomas Werner.
\newblock Constraint on five points in two images.
\newblock In {\em IEEE Conference on Computer Vision and Pattern Recognition},
  2003.

\bibitem{WernerPajdla2001}
Tom{\'a}{\v{s}} Werner and Tom{\'a}{\v{s}} Pajdla.
\newblock Cheirality in epipolar geometry.
\newblock In {\em IEEE International Conference on Computer Vision}, 2001.

\bibitem{werner2001oriented}
Tom{\'a}{\v{s}} Werner and Tom{\'a}{\v{s}} Pajdla.
\newblock Oriented matching constraints.
\newblock In {\em British Machine Vision Conference}, 2001.

\end{thebibliography}

\end{document}